\documentclass[11pt,leqno]{article}
\usepackage{a4wide}
\usepackage{amsmath}
\usepackage{amsfonts}
\usepackage{amssymb}
\usepackage{amsthm}
\usepackage{latexsym}
 \sloppypar

\newtheorem{theorem}{Theorem}[section]
\newtheorem{definition}[theorem]{Definition}

\newtheorem{lemma}[theorem]{Lemma}
\newtheorem{remark}[theorem]{Remark}
\newtheorem{proposition}[theorem]{Proposition}

\newtheorem{corollary}[theorem]{Corollary}
\newtheorem{example}[theorem]{Example}

\newcounter{exocpt}

\def\supp#1{\text{supp}(#1)}
\newcommand{\im}{\text{Im }}

\renewcommand{\H}{{\mathcal H}}

\def\C{\mathbb C}
\def\R{\mathbb R}
\def\I{\mathbb I}
\def\N{\mathbb N}
\def\al{\alpha}

\def\be{\beta}

\def\de{\delta}
\def\rh{\rho}
\def\et{\eta}

\def\ve{\varepsilon}

\def\la{\lambda}

\def\va{\varphi}
\def\ta{\tau}
\def\sp#1#2{\langle{#1},{#2}\rangle}

\def\c{\mathfrak{c}}

\def\g{\mathfrak{g}}
\def\gg{\mathfrak{g}}

\def\b{\mathfrak{b}}

\def\noi{\noindent}

\def\la{\lambda}
\def\ve{\varepsilon}

\def\ch{\chi}
\def\ta{\tau}
\def\ps{\psi}

\def\C{\mathbb{C}}

\def\Ad{\rm{\, Ad \,}}

\def\ol#1{\overline{#1}}
\def\nn{\nonumber}

\def\Re{{\mathbb R}}
\def\R{{\mathbb R}}
\def\C{{\mathbb C}}
\def\N{{\mathbb N}}

\def\Z{{\mathbb Z}}

\def\I{{\mathbb I}}

  \def\Id{{\mathbb I}}
\def\A{{\mathcal A}}
\def\B{{\mathcal B}}
\def\D{{\mathcal D}}
\def\F{{\mathcal F}}

\def\H{{\mathcal H}}
\def\K{{\mathcal K}}

\def\S{{\mathcal S}}
\def\ZZ{{\mathcal Z}}

\def\iy{\infty}

\def\ol#1{\overline{#1}}

\def\hb#1{\hbox{#1}}
\def\val#1{\vert #1\vert}

\def\no#1#2{\Vert #1\Vert_{#2} }
\def\noop#1{\Vert #1\Vert_{\rm op}}

\def\lt{L^2(\R)}

\def\exp#1{\hb{exp}(#1)}
\def\ind#1#2{\hb{ind}_{#1}^{#2}}

\def\supp#1{\text{supp}(#1)}

\def\res#1{_{\vert #1}}
\def\inv{^{-1}}

\def\es{\emptyset}

\def\me{\medskip\noindent}
\def\hb #1{\hbox{#1}}

\def\hb#1{\hbox{#1}}
\def\val#1{\vert #1\vert}

\def\im#1{\hb{im}(#1)}

\def\sp#1#2{\langle #1,#2\rangle }

\def\ca#1{{\mathcal #1}}

\def\L#1#2{L^#1(\R^{#2})}
\def\l#1#2{L^{#1}{(#2)}}

\def\lef({\left(}
\def\rig){\right)}

\begin{document}
\title{The $C^*$-algebras of the Heisenberg Group and of thread-like Lie groups.}
\author{Jean Ludwig and Lyudmila Turowska}
\date{}


 \maketitle
\begin{abstract}
We describe the $C^*$-algebras of the Heisenberg group $H_n$, $n\geq 1$, and the
thread-like Lie groups $G_N$, $N\geq 3$,   in terms of $C^*$-algebras of
operator fields.
\end{abstract}
\section{Introduction and notation}
Let $H_n$ be the Heisenberg group of dimension $2n+1$. It has been
known for a long time that the $C^*$-algebra, $C^*(H_n)$, of $H_n$ is
an extension of an ideal $J$ isomorphic to $C_0(\R^*,\K)$ with the
quotient algebra isomorphic to $C^*(\R^{2n})$, where $\K$ is the
$C^*$-algebra of compact operators on a separable Hilbert space,
$\R^*=\R\setminus\{0\}$ and $C_0(\R^*,\K)$ is the $C^*$-algebra of
continuous functions vanishing at infinity from $\R^*$ to $\K$.

We obtain an exact characterisation of this extension giving a linear
mapping from $C^*(\R^{2n})$ to $C^*(H_n)/J$ which is a cross section
of the quotient mapping $i:C^*(H_n)\to C^*(H_n)/J$. More precisely,
realizing $C^*(H_n)$ as a $C^*$-subalgebra of the $C^*$-algebra
$\F_n$ of all operator fields $(F=F(\lambda))_{\lambda\in \R}$
taking values in $\K$ for $\lambda\in \R^*$ and in $C^*(\R^{2n})$ for
$\lambda=0$, norm continuous on $\R^*$ and vanishing as $\lambda\to
\infty$, we construct a linear map $\nu$ from  $C^*(\R^{2n})$ to
$\F_n$,
such that the $C^*$-subalgebra is isomorphic to the $C^*$-algebra
$D_\nu(H_n)$ of all $(F=F(\lambda))_{\lambda\in\R}\in\F_n$ such that
$$\|F(\lambda)-\nu(F(0))\|_{\text{op}}\to 0,$$
where $\|\cdot\|_{\text{op}}$ is the operator norm on $\K$.
The constructed  mapping $\nu$ is an almost homomorphism in the sense that $$\lim_{\lambda\to 0}\|\nu(f\cdot h)(\lambda)-\nu(f)(\lambda)\circ\nu(h)(\lambda)\|_{\text{op}}=0.$$ Moreover, any such almost homomorphism $\tau:C^*(\R^2)\to \F_n$ defines a $C^*$-algebra, $D_\tau(H_n)$, which is an extension of
$C_0(\R^*,\K)$ by $C^*(\R^{2n})$. A question we left unanswered : what mappings $\tau$ give the $C^*$-algebras which are isomorphic to $C^*(H_n)$. We note that the condition $$\lim_{\lambda\to 0}\|\tau(h)(\lambda)\|_{\text{op}}=\| h\|_{C^*(\R^{2n})}, \text{ for all } h\in C^*(\R^{2n}),$$ which is equivalent to the condition that the topologies of $D_{\tau}(H_n)$ and that of $C^*(H_n)$ agree, is not the right condition: there are examples of splitting extensions of type $D_{\tau}(H_n)$ with the same spectrum as $C^*(H_n)$ (see \cite{delaroche} and Example~\ref{example_del}) while it is known that $C^*(H_n)$ is a non-splitting extension.

We note that another  characterisation of $C^*(H_n)$ as a
$C^*$-algebra of operator fields is given without proof in a short paper by
Gorbachev \cite{gorbachev}.

The second part of the paper deals with the $C^*$-algebra of
thread-like Lie groups $G_N$, $N\geq 3$. The group $G_3$ is the Heisenberg group of dimension 3 treated in the first part of the paper. The groups $G_N$  are nilpotent Lie
groups and their unitary representations can be described using the Kirillov orbit method. The topology of the dual space $\widehat{G_N}$  has been investigated in details in \cite{A.L.S.}. In particular, it was shown that like for the Heisenberg group $G_3$  the topology of $\widehat{G_N}$, $N\geq 3$ is not Hausdorff.
It is known that $\widehat{G_3}=\R^*\cup{\R^2}$ as a set with natural topology
on each pieces, the limit set when $\lambda\in \R^*$ goes to  $0$ is the whole
real plane $\R^2$. The topology of $\widehat{G_N}$, becomes more complicated
with growth of the dimension $N$. Using a description of the limit sets of
converging sequences $(\pi_k)\in\widehat{G_N}$ obtained  in
\cite{A.K.L.S.S.} and \cite{A.L.S.} we give a characterisation of the
$C^*$-algebra of $G_N$ in the spirit of one for the Heisenberg group $H_n$.
Namely, parametrising
$\widehat{G_N}$ by a set $S_N^{gen}\cup \R^2$, where $S_N^{gen}$ consists of
element $\ell\in \g_N^*$ corresponding to non-characters (here $\g_N$ is the Lie
algebra of $G_N$), we realize $C^*(G_N)$ as a $C^*$-algebra of operator fields
$(A=A(\ell))$ on $S_N^{gen}\cup\{0\}$,  such that $A(\ell)\in \K$, $\ell\in
S_N^{gen}$, $A(0)\in C^*(\R^2)$ and $(A=A(\ell))$ satisfy for each converging
sequence in the dual space the generic, the character and the infinity
conditions (see Definition~\ref{gencon}).

We shall use the following notation. $L^p(\R^n)$ denote the space of (almost everywhere equivalence classes) $p$-integrable functions for $p=1,2$ with norm $\|\cdot\|_p$.  By $\|f\|_{\infty}$ we denote the supremum norm $\sup_{x\in \Omega}|f(x)|$ of a continuous  function $f$  vanishing at infinity from a locally compact space $\Omega$ to $\C$.
 $\D(\R^n)$  is the space of complex-valued
$C^{\infty}$ functions with compact support and
$\S(\R^n)$ is the space of Schwartz functions, i.e. rapidly decreasing complex-valued $C^{\infty}$ functions on $\R^n$.
The space of Schwartz functions on the groups $H_n$ and $G_N$ (see \cite{CG}) will be denoted by $\S(H_n)$ and $\S(G_N)$ respectively.
We use the usual notation $B(H)$ for
the space of all linear bounded operators on a Hilbert space $H$ with the operator norm $\|\cdot\|_{\text{op}}$.

\bigskip

\noindent{\bf Keywords.} Heisenberg group, thread-like Lie group, unitary representation, $C^*$-algebra.
\bigskip

\noindent{\bf 2000 Mathematics Subject Classification:} 22D25, 22E27, 46L05.

\section{The $C^*$-algebra of the Heisenberg group $H_n$}

Let $H_n$ be the $2n+1$ dimensional Heisenberg group, which is
defined as to be the Lie group whose underlying variety is the
vector space $\R^n\times\R^n\times\R$ and on which the
multiplication is given by
\begin{equation}
\nonumber (x,y,t)(x',y',t')=(x+x',y+y',t+t'+\frac{1}{2} (x\cdot
y'-x'\cdot y)),
\end{equation}
where $ x\cdot y=x_1y_1+\cdots +x_ny_n $ denotes the Euclidean
scalar product on $ \R^n $.
 The center of $H_n$ is the subgroup $\ca Z:=\{0_n\}\times\{0_n\}\times \R$
 and the commutator subgroup $[H_n,H_n]$ of $H_n$ is given by
 $[H_n,H_n]=\ca Z$. The Lie algebra $\g$ of $H_n$ has the basis
 $$\B:=\{X_j,Y_j, j=1\cdots ,n,Z=(0_n,0_n,1)\},$$
where $ X_j=(e_j,0_n,0),Y_j=(0_n,e_j,0),j=1, \cdots ,n $ and  $
e_j $ is the j'th canonical basis vector of $ \R^n $,
 with the non trivial brackets
 \begin{equation}
 \nonumber [X_i,Y_j]=\de_{i,j}Z.
\end{equation}

\subsection{The unitary dual of $H_n$.}\label{unit}

\medskip
 The unitary dual $\widehat H_n$ of $H_n$ can be described
 as follows.
\medskip
\subsubsection{The infinite dimensional irreducible
representations}
\medskip
 For every $\la\in\R^*$, there exists a unitary representation $\pi_\la$
of $H_n$ on
 the Hilbert space $L^2(\R^n)$, which is given by the formula

 \begin{equation}
 \nonumber \pi_\la(x,y,t)\xi(s):=e^{{-2\pi i}  \la t{-2\pi i}   \frac{\la} 2
x\cdot y{+2\pi i}\la
 s\cdot y}\xi(s-x),\ s\in\R^n,\xi\in L^2(\R^n),(x,y,t)\in H_n.
 \end{equation}

It is easily seen that $\pi_\la$ is in fact irreducible and that
$\pi_\la$ is  equivalent to $\pi_\nu$ if and only if $\la=\nu$.

The representation $\pi_\la$ is  equivalent to the induced
representation $\ta_\la:=\ind P {H_n} \ch_\la$, where
$P=\{0_n\}\times \R^n\times\R$ is a polarization at the linear
functional $\ell_\la((x,y,t)):=\la t, (x,y,t)\in \g$ and where
$\ch_\la$ is the character of $P$ defined by
$\ch_\la(0_n,y,t)=e^{{-2\pi i}  \la t}$.

The theorem of Stone-Von Neumann tells us that every infinite dimensional
unitary representation of $ H_n $ is equivalent to one of the   $ \pi_\la $'s.
(see \cite{CG}).
\subsubsection{The finite dimensional irreducible
representations}{}
\medskip
Since $ H_n $ is nilpotent, every irreducible finite dimensional
representation of $ H_n $ is one-dimensional, by Lie's theorem.

Any one-dimensional representation is  a unitary
character $\ch_{a,b}$,  $(a,b)\in \R^n\times\R^n$,  of $H_n$, which is given by
\begin{equation}
\nonumber \ch_{a,b}(x,y,t)=e^{{-2\pi i}   (a\cdot  x+b\cdot y)}, (x,y,t)\in
H_n.
\end{equation}
 For $ f\in L^1(H_n) $, let
\begin{equation}
 \nn \hat f(a,b):=\ch_{a,b}(f)=\int_{H_n}f(x,y,t)e^{{-2\pi i}  (x\cdot a+y\cdot
b)}dxdydt,\ a,b\in \R^n,
\end{equation}
and
$$ \no f{\iy,0}:=\sup_{a,b\in
\R^{n}}\val{\ch_{a,b}(f)}=\no {\hat f}\iy .$$

\subsection{The topology of $\widehat {C^*(H_n)}$}

Let $C^*(H_n)$ denote the full $C^*$-algebra of $H_n$. We recall
that $C^*(H_n)$ is obtained by the completion of $L^1(H_n)$ with
respect to the norm
$$\no f{C^*(H_n)}=\sup \no{\int
f(x,y,t)\pi(x,y,t)dxdydt}{\text{op}},$$ where the supremum is
taken over all unitary representations $\pi$ of $H_n$.
\begin{definition}\label{pi0}

\rm  Let
$$ \rho=\ind \ZZ{H_n}1 $$ be the left regular representation of $
G_N $ on the Hilbert space $ L^2(G_N/\ZZ) $. Then the image $
\rho(C^*(H_n)) $ is just the $ C^*$-algebra of $ \R^{2n} $ considered
as an algebra of convolution operators on $ L^2(\R^{2n}) $ and $ \rho
(C^*(H_n))$ is isomorphic to the algebra $ C_0(\R^{2n}) $ of continuous
functions vanishing at infinity on $ \R^{2n} $ via the Fourier transform.
For $f\in L^1(H_n)$ we have $\widehat{\rho(f)}(a,b)=\hat f(a,b,0)$, $a$, $b\in\R^{n}$.

 \end{definition}

\begin{definition}\label{c3hn
} \rm
Define for $ C^*(H_n) $  the \textit{Fourier transform} $ F(c)$ of $ c $ by
$$ F(c)(\la):=\pi_\la(c)\in B(L^{2}(\R^n)),\la\in\R^* $$and
$$ F(c)(0):=\rho(c)\in C^*(\R^{2n}).$$
 \end{definition}

\subsubsection{Behavior on $ \R^* $}

As for the topology of the dual space, it is well known that
$[\pi_\la]  $ tends to $[\pi_\nu]  $ in $\widehat H_n$ if and only
if $\la$  tends to $\nu$ in $\R^*$, where $[\pi]$ denotes the unitary
equivalence class of the unitary representation $\pi$. Furthermore, if $ \la $ tends
to 0, then the representations $ \pi_\la $  converge in the dual
space topology to all the characters $ \ch_{a,b}, a,b\in\R^n $.

Let us compute  for $f\in L^1(H_n)$ the operator $\pi_\la(f)$. We
have for $\xi\in L^2(\R^n)$ and $s\in\R^n$ that
\begin{eqnarray}\label{pila}
\nonumber
\pi_\la(f)\xi(s)&=&\int_{\R^n\times\R^n\times\R}f(x,y,t)\pi_\la(x,y,
t)\xi(s)dxdydt \\
\nonumber &=&\int_{\R^n\times\R^n\times\R}f(x,y,t)e^{{-2\pi i}  \la
t-\frac{2\pi i\la}{2}x\cdot y {+2\pi i}\la s\cdot
y}\xi(s-x) dxdydt \\
 &=&\int_{\R^n\times\R^n\times\R}f(s-x,y,t)e^{{-2\pi i}  \la
t-\frac{2\pi i\la}{2}(s-x)\cdot y {+2\pi i}\la s\cdot
y}\xi(x)dx dydt  \\
\nonumber &=&\int_{\R^n}\hat
f^{2,3}(s-x,-\frac{\la}{2}(s+x),\la)\xi(x)dx.
\end{eqnarray}

Here
$$\hat f^{2,3}(s,u,\la)=\int_{\R^n\times\R}f(s,y,t)e^{{-2\pi i}  ( y\cdot
u+\la t) }dydt, \ (s,u,\la)\in \R^n\times\R^n\times\R$$ denotes
the partial Fourier transform of $f$ in the variables $y$ and $t$.

Hence $\pi_\la(f)$ is a kernel operator with kernel
\begin{equation}\label{kernel}
 f_\la(s,x):=\hat f^{2,3}(s-x,-\frac{\la}{2}(s+x),\la),\
s,x\in\R^n.
\end{equation}

If we take now a Schwartz-functions $f\in \S(H_n)$, then the
operator $\pi_\la(f)$ is  Hilbert-Schmidt and its Hilbert-Schmidt
norm $\no{\pi_\la(f)}{\text{H.S.}} $ is given by
\begin{equation}\label{hilbsch}
\no{\pi_\la(f)}{\text{H.S.}}^2=\int_{\R^2}\val{f_\la(s,x)}^2dxds=\int_{\R^2}\val
{\hat f^{2,3}(s,{\la}x,\la)}^2dsdx<\iy.
\end{equation}

\begin{proposition}\label{convinnorm}
For any $ c\in C^*(H_n)$ and $ \la{\in\R^*} $, the operator $
\pi_\la(c)  $ is compact , the mapping $ \R^*\to B(L^2(\R^n)):
\la\mapsto \pi_\la(c)$ is norm continuous and tending to 0 for $
\la $ going to infinity.
 \end{proposition}
\begin{proof}
Indeed, for $ f\in\S(H_n) $, the compactness of the operator $
\pi_\la(f) $ is a consequence of (\ref{hilbsch}) and by (\ref{pila}) we have the
estimate:
\begin{eqnarray}
\nonumber
\no{\pi_\la(f)-\pi_\nu(f)}{H.S}^2 &=&  \int_{\R^n\times\R^n}\val{\hat
f^{2,3}(s-x,-\frac{\la}{2}(s+x),\la)
-\hat f^{2,3}(s-x,-\frac{\nu}{2}(s+x),\nu) }^2 ds dx\\
\nonumber &=&  \int_{\R^n\times\R^n}\val{\hat
f^{2,3}(s,{\la}x,\la) -\hat f^{2,3}(s,{\nu}x,\nu) }^2 ds dx
\end{eqnarray}
Hence, since $ f $ is a Schwartz function, this expression goes to
0 if $ \la $ tends to $ \nu $ by Lebesgue's theorem of dominated
convergence. Therefore the mapping $ \la\mapsto \pi_\la(f)$ is
norm continuous. Furthermore, the Hilbert-Schmidt  norms of the
operators $ \pi _\la(f)$ go to 0, when $ \la  $ tends to infinity.
The proposition follows from the density of $ \S(H_n) $ in $ C^*(H_n)
$.
 \end{proof}

\subsubsection{Behavior in 0}
Let us now see the behavior of $\pi_\la(f)$ for Schwartz
functions $f\in \S(H_n)$, as $\la$ tends to 0.

 Choose a Schwartz-function $\et$ in $\S(\R^n)$ with $L^2$-norm
equal to 1. For $u=(a,b)$ in $\R^n\times\R^n,\ \la \in\R^*$, we
define the function $\et(\la,a,b)$ by
\def\laa{\val{\la}^{1/2}}\def\laaa{\val{\la}^{1/4}}
\def\lav{\val{\la}}
\def\laaa{\val{\la}^{1/4}}
\def\laai{\val{\la}^{-1/2}}\def\laaai{\val{\la}^{-1/4}}
\begin{equation}\label{etalau}
 \et(\la,a,b)(s):=\lav^{n/4} e^{ 2\pi i
a\cdot s}\et(\val{\la}^{1/2}(s+\frac{b }{\la} ))\
s\in\R^n.
\end{equation}
and let
$\et_\la(s)=|\la|^{n/4}\et(|\la|^{1/2}s), \ s\in\R^n$.

 Let us compute
\begin{eqnarray}\label{cuu'}
 \nonumber c_{\la,u,u'}(x,y,t)&=&\sp{\pi_\la(x,y,t)\et(\la,u)}{\et(\la,u')}\\
\nonumber &=&\int_{\R^n} e^{{-2\pi i}  \la t{-2\pi i}  (\la/2) x\cdot
y}e^{2\pi i\la
s\cdot y}\et(\la,u)(s-x)\ol{\et(\la,u')(s)}
ds  \\
\nonumber &=& \lav^{n/2} e^{{-2\pi i}  \la t {-2\pi i}  (\la/2) x\cdot
y}\int_{\R^n}
e^{2\pi i\la s\cdot y{-2\pi i}  a\cdot x }e^{2\pi i(a-a')\cdot s}\\
&&\et(\lav^{1/2}(s-x+\frac{b}{\la}
))\ol{\et(\lav^{1/2}(s+\frac{b'}{\la}))}
ds      \\
\nn &=& \lav^{n/2} e^{{-2\pi i}  \la t {-2\pi i}  (\la/2) x\cdot
y}e^{{-2\pi i} b\cdot
y}e^{{-2\pi i}  a\cdot x}\int_{\R^n} e^{2\pi i\la s\cdot y}e^{2\pi i(a-a')\cdot
(s-\frac{b}{\la})}\\
\nn&&\et(\lav^{1/2}(s-x)\ol{\et(\lav^{1/2}(s+\frac{b'-b}{\la}))}
ds.
\end{eqnarray}

Hence for $u=u'$ we get
 \begin{eqnarray}
  \nonumber c_{\la,u,u}(x,y,t)&=& e^{{-2\pi i}  \la t {-2\pi i}  \frac{\la}{2}
x\cdot y}e^{{-2\pi i}  a\cdot x {-2\pi i}   b\cdot y}\int_{\R^n} e^{2\pi i ({\rm
sign}{\la})\lav^{1/2} s\cdot
y}\et(s-\lav^{1/2} x )\ol{\et(s)} ds\\
 \nonumber &\to&e^{{-2\pi i}  a\cdot x {-2\pi i}
b\cdot y}\int_{\R^n} \et(s )\ol{\et(s)}ds=e^{{-2\pi i}  a\cdot x {-2\pi i}
b\cdot
y}.
\end{eqnarray}

 It follows also
that the convergence of the coefficients $ c_{\la,u,u} $ to the
characters $ \ch_{a,b} $ is uniform in $u$ and uniform on
compacta in $(x,y,t)$ since
\begin{eqnarray}
\nonumber \vert c_{\la,u,u}(x,y,t)-\ch_{a,b}(x,y,t)\vert&=&\vert
\int_{\R^n} ( e^{{-2\pi i}  \la t {-2\pi i}  \frac{\la}{2} x\cdot y}e^{2\pi i
({\rm
sign}\la)\laa s\cdot y }\\
\nn &&\et(s-\lav^{1/2} x
)\ol{\et(s)}-\val{\et(s)}^2)ds\vert \to 0\\
\nonumber &&\text{ as } \la\to 0.
\end{eqnarray}


\begin{proposition}\label{conforetlaab}
For every $ u=(a,b)\in\R^{n}\times\R^n, c\in C^*(H_n) $, we have
that
\begin{equation}
 \nn \lim_{\la\to 0} \no{ F(c)(\la)\et(\la,u)-
 \widehat{F(c)(0)}(u)\et(\la,u)}2=0
\end{equation}
uniformly in $ (a,b) $.
 \end{proposition}
\begin{proof}
For $ c\in C^*(H_n) $ we have  that

$\no{ F(c)(\la)\et(\la,u)-\widehat{F(c)(0)}(u)\et(\la,u)}2^2=\no{ \pi_\la(c)\et(\la,u)-\chi_{a,b}(c)\et(\la,u)}2^2=$
\begin{eqnarray}
\nonumber &=& \langle{\pi_\la(c)\et(\la,u)-\chi_{a,b}(c)\et(\la,u)
},{\pi_\la(c)\et(\la,u)-\chi_{a,b}(c)\et(\la,u)}\rangle\\
\nonumber &=&\langle{\pi_\la(c^*\ast
c)\et(\la,u)},{\et(\la,u)}\rangle -
\ol{\chi_{a,b}(c)}\langle{\pi_\la( c)\et(\la,u)},{\et(\la,u)}\rangle
\\
\nonumber &-& \chi_{a,b}(c)\ol{\langle{\pi_\la(
c)\et(\la,u)},{\et(\la,u)}\rangle}+\vert \chi_{a,b}(c)\vert^2\\
\nonumber &\to &  \vert \chi_{a,b}(c)\vert^2-\vert \chi_{a,b}(c)\vert^2-\vert \chi_{a,b}(c)\vert^2+
\vert \chi_{a,b}(c)\vert^2=0.
\end{eqnarray}

 \end{proof}

\subsection{A $C^*$-condition}\label{matrix}
The aim of this section is to obtain  a characterization of the  $C^*$-algebra
$C^*(H_n)$ as a $C^*$-algebra of operator fields (\cite{Lee1, Lee2}).

 Let us first define a larger $C^*$-algebra $\mathcal{F}_n$.
\begin{definition}
\rm Let $\mathcal F_n$  be the family  consisting of all operator
fields $(F=F(\la))_{\la\in\R}  $ satisfying the following
conditions:
\begin{enumerate}
\item $F(\la)$ is a compact operator on $L^2(\R^n)$ for every
$\la\in\R^*$, \item $ F(0)\in C^*(\R^2)$, \item the mapping
$\R^*\to B(L^2(\R^n)): \la\mapsto F(\la)$ is norm continuous,
\item $ \lim_{\la\to\iy}\noop {F(\la)}=0$.
 \end{enumerate}
\end{definition}
\begin{proposition}
$ \F_n $ is a $C^*$-algebra.
\end{proposition}
\begin{proof} The proof is straight forward.
\end {proof}

\begin{proposition}\label{inject}
The Fourier transform  $F: C^*(H_n)\to \mathcal F_n$  is an injective
homomorphism.
\end{proposition}

\begin{proof}
It is clear from the definition of $F$ and Proposition~\ref{convinnorm}
that $ F $ is  a homomorphism with values in $\F_n$. If $ F(c)=0 $, then for
each irreducible representation $ \pi $ of $ C^*(H_n) $, $\pi(c)=0
$. Hence $ c=0 $.
 \end{proof}

\begin{lemma}\label{xi}
Let $\xi\in \S(\R^{2n})$. Then, for any $ \la\in\R^*, $
\begin{equation}
\nonumber
\xi=\frac{1}{\lav^{n}}\int_{\R^n\times\R^n}\sp{\xi}{\et(\la,u)}\et(\la,u)du,
\end{equation}
where $ \et({\la,u}) $ is as  in Definition \ref{etalau},
the integral converging in $L^2(\R^n)$.

\end{lemma}

\begin{proof}

Let $ \xi\in \S(\R^n). $
 Then
\begin{eqnarray}{}
\nonumber &&
\int_{\R^n\times\R^n}\sp{\xi}{\et(\la,a,b)}\et(\la,a,b)(x)dadb\\
\nonumber &=& \int_{\R^n\times\R^n}\left(\int_{\R^n}\xi(s)e^{{-2\pi i}
a\cdot s}\ol{\et_\la(s+\frac{b}{\la} ))}ds\right)
 e^{2\pi i a\cdot x}{\et_\la(x+\frac{b}{\la} )} dadb  \\
\nonumber &&(\text{by Fourier's inversion formula} )  \\
\nonumber &=&
\int_{\R^n\times\R^n}\xi(x)\ol{\et_\la(x+\frac{b}{\la} )}
\et_\la(x+\frac{b}{\la} )db =\lav^{n} \xi(x)
\end{eqnarray}
giving
 $\xi=\frac{1}{\lav^n}
\int_{\R^n\times\R^n}\sp{\xi}{\et(\la,a,b)}\et(\la,a,b)da db$.

Furthermore, since $\xi$ is a Schwartz function, it follows that
the mapping
$$(a,b)\to \sp{\xi}{\et(\la,a,b)}=\lav^{n/4}
\int_{\R^n}\xi(s)e^{-2\pi i a\cdot s}\overline{\et(\lav^{1/2}(s+\frac{b}{\la} ))}ds $$
 is
also a Schwartz function in the variables $a,b$. Hence the
integral $\int_{\R^n\times\R^n}\sp{\xi}{\et(\la,a,b)} \et(\la,a,b)da
db$ converges in $\S(\R^n)$ and hence also in $L^2(\R^n)$.

\end{proof}

\begin{remark}\rm
By Lemma~\ref{xi},
$$\pi_\la(f)\xi=\frac{1}{|\la|^n}\int_{\R^{2n}}\pi_{\la}(f)\et(\la,u)
\langle\xi,\et(\la,u)\rangle
du=\frac{1}{|\la|^n}\int_{\R^{2n}}\pi_{\la}(f)\circ P_{\et(\la,u)}\xi du$$
for any $f\in C^*(H_n)$, where $P_{\et(\la,u)}$ is the orthogonal
projection onto the one dimensional subspace $\C\et(\la,u)$.

\end{remark}

\begin{definition}\label{projint}
\rm   For a vector $ 0\ne \et\in\L2n $, we let $ P_\et $ be the
orthogonal projection onto the one dimensional subspace $ \C \et
$.

  Define for $ \la\in\R^* $ and $ h\in C^*({\R^{2n}}) $ the
linear operator
\begin{eqnarray}\label{inthxi}
  \nu_\la(h) &:=&\int_{\R^{2n}} \hat h(u) P_{\et(\la,u)}\frac{du}{\lav^n}.
\end{eqnarray}

 \end{definition}
\begin{proposition}\label{intisbounded}$  $
\begin{enumerate}
\item For every $ \la\in\R^* $ and $ h\in \S({\R^{2n}}) $ the
integral (\ref{inthxi}) converges in operator norm.
\item   $ \nu_\la(h) $ is compact and
$\|\nu_\la(h)\|_{op}\leq\|h\|_{C^*({\mathbb R}^{2n})}$.
\item  The mapping $ \nu_\la: C^*(\R^{2n})\to \F_n $
is involutive, i.e. $ \nu_\la(h^*)=\nu_\la(h)^* , h\in C^*(\R^{2n})$, where by
$\nu_\la$ we denote also  the extension of $\nu_\la$ to $C^*(\R^{2n})$.
 \end{enumerate}

 \end{proposition}

\begin{proof}

Since $ \noop{P_{\et(\la,u)}}=\no{\et(\la,u)}2^2=1 $, we have that
\begin{eqnarray}
 \nn \noop{\nu_\la(h)}
&=&\noop{\int_{\R^{2n}}\hat
h(u)P_{\et(\la,u)}\frac{du}{\lav^n}}\leq \int_{\R^{2n}}\val {\hat
h(u)}\frac{du}{\lav^n}=\frac{\no {\hat h}1}{\lav^n}.
\end{eqnarray}
Hence the integral $ \int_{\R^{2n}}\hat
h(u)P_{\et(\la,u)}\frac{du}{\lav^n} $ converges in operator norm
for $ h\in\S(\R^{2n}) $.

We compute $ \nu_\la(h) $ applied to a Schwartz function $
\xi\in\S(\R^n) $:
\begin{eqnarray}\label{nula}
\nonumber {\nu_\la(h)\xi}(x)&=& \int_{\R^{2n}}\hat
h(u)\langle{\xi},{\et(\la,u)}\rangle \et(\la,u)(x)\frac{du}{\lav^n} \\
\nonumber &=&  \int_{\R^{2n}}\hat
h(u)\left(\int_{\R^n}{\xi(r)}{\ol\et_\la(r+\frac{b}{\la})}e^{{-2\pi i}   a
\cdot r}
dr\right)e^{{2\pi i a}\cdot x} \et_\la(x+\frac{b}{\la})\frac{da db}{\lav^n}\\
 &=&   \int_{\R^{n}}\hat h^2(-,b)\ast (\xi
\ol\et_{\la,b})(x)
\et_\la(x+\frac{b}{\la})\frac{db}{\lav^n}\\
\nn&&(\text{where } \et_{\la,b}(s):=\et_\la(s+\frac{b}{\la}),s\in\R^n)\\
\nonumber &=&  \int_{\R^{2n}} \int_{\R^n}\hat h^2(x-s,b) \xi(s)
{\ol\et_{\la}(s+\frac{b}{\la})}
\et_\la(x+\frac{b}{\la})\frac{db}{\lav^n}ds\\
\nn &=& \int_{\R^{2n}} \int_{\R^n}\hat h^2(x-s,|\la|^{1/2}b) \xi(s)
{\ol\et(|\la|^{1/2}s+{\rm sign}\la \cdot b)}
\et(|\la|^{1/2}x+{\rm sign}\la \cdot b)dbds
\end{eqnarray}

The kernel function $  h_\la(x,s) $ of $\nu_\la(h)$ is
in $ \S(\R^{2n}) $ if $ h\in\S(\R^{2n}) $. In particular  $
\nu_\la(h) $ is a compact operator and we have the following
estimate for the Hilbert-Schmidt norm, $\|\cdot\|_{H.S}$, of $ \nu_\la(h) $:
 \begin{eqnarray}
\nonumber \no{\nu_\la(h)}{H.S}^2&=&
\int_{\R^n\times\R^n}\val{\int_{\R^n}\hat
h^2(x-s,b)\ol\et_\la(s+\frac{b}{\la})\et_\la(x+\frac{b}{\la})
\frac{db}{\lav^n}}^2dsdx\\
\nonumber &\leq&  \int_{\R^{3n}}\val {\hat
h^2(x-s,\la(b-x))}^2\vert\et_\la(s-x+b)\vert^2dbdxds\\
\nonumber &=&  \int_{\R^{3n}}\val {\hat
h^2(x,\la(b+s))}^2\vert\et_\la( b)\vert^2dbdxds\\
\nn&=&\int_{\R^{2n}}\val {\hat h^2(x,\la s)}^2dxds<\iy.
\end{eqnarray}

 Let us show  that $
\noop{\nu_\la(h)}\leq\no{\hat h} {\iy}$. Indeed
\begin{eqnarray}
\nonumber \no{\nu_\la(h)\xi}2^2&=& \int_{\R^n}\val{\int_{\R^n}\hat
h^2(-,b)\ast (\xi \ol\et_{\la,b})(x)
\et_\la(x+\frac{b}{\la})\frac{db}{\lav^n}}^2dx \\
\nonumber &\leq&
\frac{1}{\lav^{2n}}\int_{\R^n}\int_{\R^n}\vert\hat h^2(-,b)\ast
(\xi \ol\et_{\la,b})(x)\vert^2
{db}dx\\
\nonumber &\leq&\frac{\no{\hat
h}\iy^2}{\lav^{n}}\int_{\R^n}\Vert\xi \et_{\la,b}\Vert_2^2
{db}\\
\nonumber &=&  \frac{\no{\hat
h}\iy^2}{\lav^{n}}\int_{\R^n}\int_{\R^n}\vert\xi(x)
\et_{\la}(x+\frac{b}{\la})\vert^2dx
{db}\\
\nonumber &=&  \no{\hat h}\iy^2\no\xi2^2.
\end{eqnarray}

Let $h\in S(\R^{2n})$. Then $\overline{\hat h}=\hat{h^*}$.
This gives
\begin{eqnarray*}
\nu_\la(h)^*=(\int_{\R^{2n}}\hat h(u)P_{\et(\la,u)}\frac{du}{|\la|^n})^*=
\int_{\R^{2n}}\overline{\hat h(u)}P_{\et(\la,u)}\frac{du}{|\la|^n}\\=
\int_{\R^{2n}}\hat {h^*}(u)P_{\et(\la,u)}\frac{du}{|\la|^n}=\nu_\la(h^*).
\end{eqnarray*}

 \end{proof}

\begin{theorem}\label{characterisation}
Let $a\in C^*(H_n)$ and let $A$ be the operator field $ A=F(a)$, i.
e.
$$ A(\la)={\pi_\la(a)},\la\in\R^*, A(0)=\rho(a)\in C^*(\R^{2n}). $$
Then $$\lim_{\la\to 0}\noop{A(\la)-\nu_\la(A(0))}=0.$$
\end{theorem}

\begin{proof}
Let $f\in \S(H_n)$, $\xi\in L^2(\R^n)$, $\eta\in {\cal S}(\R^n)$,
$\|\eta\|_2=1$. Then by (\ref{pila}) and (\ref{nula})
\begin{eqnarray*}
((\pi_\la(f)-\nu_\la(\rho(f))\xi)(x)&=&\int_{\R^n}\hat{f}^{2,3}(x-s,-\frac{\la}{2}
(x+s),\la)\xi(s)ds\\
&-&\int_{\R^n}\int_{\R^n}\hat{f}^{2,3}(x-s,b,0)\xi(s)\ol\eta_\la(s+\frac{b}{\la}
)\eta_{\la}(x+\frac{b}{\la})\frac{db}{|\la|^n}ds\\
&=&\int_{\R^n}\int_{\R^n}\hat{f}^{2,3}(x-s,-\frac{\la}{2}(x+s),
\la)\eta_\la(b)\bar{\eta}_\la(b)\xi(s)dbds\\
&-&\int_{\R^n}\int_{\R^n}\hat{f}^{2,3}(x-s,\la(b-x),
0)\xi(s)\ol\eta_\la(s-x+b)\eta_{\la}(b)dbds.
\end{eqnarray*}

Let
\begin{eqnarray*}
 u_{\la}(x,b)&=& \int_{\R^n}
\xi(s)\ol\et_\la(s-x+b) (\hat f^{2,3}(x-s,-\frac{\la}{2}(x+s), \la)
-\hat f^{2,3}(x-s,-\frac{\la}{2}(x+s),
0))ds,\\
v_\la(x,b)&=&\int_{\R^n}
\xi(s)\ol\et_\la(s-x+b)(\hat f^{2,3}(x-s,-\frac{\la}{2}(x+s), 0)-\hat
f^{2,3}(x-s,\la(b-x),0))ds
\end{eqnarray*}

and
$$w_\la(x)=\int_{\R^n}\int_{\R^n}\hat{f}^{2,3}(x-s,-\frac{\la}{2}(x+s),
\la)\xi(s)\eta_\la(b)(\ol\eta_\la(b)-\ol\eta_\la(s-x+b))dbds.$$

We have
\begin{equation}
((\pi_\la(f)-\nu_\la(\rho(f))\xi)(x)=\int_{\R^n}u_\la(x,b)\eta_\la(b)db+
\int_{\R^n}v_\la(x,b)\eta_\la(b)db+w_\la(x).
\end{equation}
Thus to prove $\noop{\pi_\la(f)-\nu_\la(\rho(f))}\to 0$ as $\la\to 0$ it is enough to
show that $\|u_\la\|_2\leq \delta_\la\|\xi\|_2$, $\|v_\la\|_2\leq
\omega_\la\|\xi\|_2$
and $\|w_\la\|_2\leq \epsilon_\la\|\xi\|_2$, where $\delta_\la$, $\omega_\la$,
$\epsilon_\la\to 0$ as $\la\to 0$.

We have
\begin{eqnarray}
 \nn \hat
 f^{2,3}(x-s,-\frac{\la}{2}(x+s)),
\la)-\hat f^{2,3}(x-s,-\frac{\la}{2}(x+s),0) =
 \la \int_0^1  \partial_3
\hat f^{2,3}(x-s,-\frac{\la}{2}(x+s)), t\la)dt
\end{eqnarray}
and
\begin{eqnarray}
 \nn \hat
f^{2,3}(x-s,-\frac{\la}{2}(x+s)), 0)-\hat f^{2,3}(x-s,\la(b-x),0)
  =\la(\frac{1}{2}(s
-x)-(s-x+b))\\
\nn  \times\int_0^1 \partial_2\hat f^{2,3}(x-s,\la(b-x) +t
(\la(\frac{1}{2}(s-x)-(s-x+b)), 0)dt.
\end{eqnarray}
Hence, since $ f\in\S(H_n) $, there exists a constant $ C>0 $ such
that
\begin{eqnarray}
 \nn  \val{f^{2,3}(x-s,-\frac{\la}{2}(x+s)),
\la)-\hat f^{2,3}(x-s,-\frac{\la}{2}(x+s),0) }
\leq\lav \frac{C}{
(1+\Vert{x-s}\Vert)^{2n+1} },
\end{eqnarray}

 and
\begin{eqnarray}
 \nn  \val{\hat
f^{2,3}(x-s,-\frac{\la}{2}(x+s)), 0)-\hat f^{2,3}(x-s,\la(b-x),0)
}
\end{eqnarray}
\begin{eqnarray}
 \nn  &\leq&\lav (\Vert{s-x+b}\Vert+\Vert{s
-x}\Vert)
\frac{C}{(1+\Vert{x-s}\Vert)^{4n+1}}
\end{eqnarray}
for all $\la\in\R^*$, $x,s\in\R^n$.
Therefore we see that
\begin{eqnarray}
\nonumber \no{u_\la}2^2&=& \int_{\R^n\times\R^n}\val{u_\la(x,b)}^2dxdb \\
\nn&\leq&\int_{\R^n\times\R^n}\left(\int_{\R^n} \vert
\xi(s)\et_\la(s-x+b)\vert\lav
\frac{C}{(1+\Vert{x-s}\Vert)^{2n+1}}ds\right)^2 dx
db\\
\nonumber &\leq& \lav^2 C' \int_{\R^{3n}}
\frac{\val{\xi(s)}^2}{(1+\Vert{x-s}\Vert)^2}\vert\et_\la(s-x+b)\vert^2
db dx ds
\\
\nonumber &\leq& C''\lav^2 \no\xi 2^2.
\end{eqnarray}
Similarly
\begin{eqnarray}
\nonumber \no{v_\la}2^2&=& \int_{\R^n\times\R^n}\val{v_\la(x,b)}^2dxdb \\
\nn&\leq&\int_{\R^n\times\R^n}\left(\int_{\R^n}\vert
\xi(s)\et_\la(s-x+b)\vert\lav (\Vert{s-x+b}\Vert+\Vert{s -x}\Vert)\right.\\
\nonumber&&\left.\frac{C}{(1+\Vert{x-s}\Vert)^{4n+1}}ds\right)^2dbdx\\
\nn&\leq&C'\int_{\R^{3n}}\vert \xi(s)\et_\la(s-x+b)\vert^2\lav^2
(\Vert{s-x+b}\Vert+\Vert{s -x}\Vert)^2\\
\nonumber&&
\frac{1}{(1+\Vert{x-s}\Vert)^{4n+1}}dsdbdx\\
\nn&\leq&2C'\int_{\R^{3n}}\vert \xi(s)\lav^{n/4}\et(\lav^{1/2}
(s-x+b))\vert^2\lav^2 \Vert{s-x+b}\Vert^2
\frac{dsdbdx}{(1+\Vert{x-s}\Vert)^{4n+1}}\\
\nn&+&2C'\int_{\R^{3n}}\vert \xi(s)\lav^{n/4}\et(\lav^{1/2}
(s-x+b))\vert^2\lav^2 \Vert{s -x}\Vert^2
\frac{dsdbdx}{(1+\Vert{x-s}\Vert)^{4n+1}}\\
\nn&\leq&2C'\lav\int_{\R^{3n}}\vert
\xi(s)\lav^{n/4}\tilde\et(\lav^{1/2} (s-x+b))\vert^2
\frac{dsdbdx}{(1+\Vert{x-s}\Vert)^{4n+1}}\\
\nn&+&2C'\lav^2\int_{\R^{3n}}\vert \xi(s)\lav^{n/4}\et(\lav^{1/2}
(s-x+b))\vert^2
\frac{dsdbdx}{(1+\Vert{x-s}\Vert)^{4n-1}}\\
\nonumber &\leq& C''\lav(\no{\tilde\et}2^2+\lav\no{\et}2^2)\no\xi
2^2,
\end{eqnarray}
for some constants $ C',C''>0 $, where the function $ \tilde\et $
is defined by $ \tilde \et(s):=\|s\|\et(s),\ s\in\R $.

Since $\eta\in{\cal S}(\R^n)$, we can use the same arguments to see that
\begin{eqnarray*}
\|w_\la\|_2^2&=&\int_{\R^n}|w_\la(x)|^2dx\\
&\leq&
\int_{\R^n}\Big(\int_{\R^n}\int_{\R^n}|\xi(s)|\eta_\la(b)||\la|^{n/4+1/2}
(\|s-x\|)\frac{C}{(1+\|x-s\|)^{4n+1}}dbds\Big)^2dx\\
&\leq&
C'|\la|^{n/2+1}\int_{\R^{3n}}|\xi(s)|^2|\eta_\la(b)|^2\frac{\|x-s\|^2dsdxdb}{
1+\|x-s\|^{4n+1}}\leq C''|\la|^{n/2+1}\|\xi\|_2^2\|\eta\|_2^2.
\end{eqnarray*}

We have proved therefore $\|\pi_\la(f)-\nu_\la(\rho(f))\|\to 0$ as
$\la\to 0$ for $f\in \S(H_n)$. Since $\S(H_n)$ is dense in
$C^*(H_n)$, the statement holds for any $a\in C^*(H_n)$.

\end{proof}

\begin{definition} \label{nuconcrete} For $\et\in \S(\R^n)$
we define the linear mapping $ \nu_\et:=\nu: C^*(\R^{2n})\to \F_n
$ by
\begin{equation}
 \nonumber \nu(h)(\la)=\nu_\la(h), \la\in\R^* \text{ and } \nu(h)(0)=h.
\end{equation}

\end{definition}
\begin{proposition}\label{nufhei}
The mapping $ \nu:C^*(\R^{2n})\to \F_n $ has the following
properties:
\begin{enumerate}
\item $ \|\nu\| =1 $. \item  For every $ h,h'\in
C^*(\R^{2n})$, we have that $$ \lim_{\la\to 0}\noop
{\nu_\la(h\cdot h')-\nu_\la(h)\circ \nu_\la(h')}=0 $$ and also
\begin{equation}
 \nonumber \lim_{\la\to 0}\noop{\nu_\la (h^*)-\nu_\la(h)^*}=0.
\end{equation}
\item  For $ (a,b)\in\R^{2n} $ and $ h\in C^*(\R^{2n}) $  we have
that
$$\lim_{\la\to
0}\no{\nu(h)(\la)\et(\la,a,b)-\hat h(a,b)\et(\la,a,b)}2=0.
$$
\item
$\lim_{\lambda\to 0}\|\nu(h)(\la)\|=\|\hat h\|_{\infty}.$
 \end{enumerate}

 \end{proposition}
\begin{proof}
 (1) follows from Proposition \ref{intisbounded}.

To prove  (2) we take  for $ h,h'\in \S(\R^{2n}) $ two
elements $ f,f'\in \S(H_n) $, such that $ \rho(f)=h, \rho(f')=h' $. Then $
\rho(f\ast f')=h\cdot h' $ and
\begin{eqnarray}
 \nn \noop{\nu_\la(h\cdot h')-\nu_\la(h)\circ \nu_\la(h')} &\leq&
\noop{\nu_\la(h\cdot h')-\pi_\la(f\ast
f')}\\\nn&+&\noop{\nu_\la(h)\circ \nu_\la(h')-\pi_\la(f)\circ \pi_\la( f')}\\
\nn  &\leq&\noop{\nu_\la(h\cdot h')-\pi_\la(f\ast f')}\\&+&
\no {f'}{C^*(H_n)}\noop{\nu_\la(h)-\pi_\la(f)}\\
\nn &+&\no{h}{C^*(\R^{2n})}\noop{\nu_\la(h')-\pi_\la(f')}.
\end{eqnarray}
Hence, by Theorem \ref{characterisation}, $ \lim_{\la\to
0}\noop{\nu_\la(f\ast f')-\nu_\la(f)\circ \nu_\la(f')}=0 $.
Furthermore
\begin{eqnarray}
\nonumber \noop{\nu_\la (h^*)-\nu_\la(h)^*}&\leq& \noop{\nu_\la
(h^*)-\pi_\la(f^*)}+\noop{\nu_\la (h)^*-\pi_\la(f)^*}\to 0\\
\nonumber &&  \text{ as }\la\to 0.
\end{eqnarray}

We conclude by the usual approximation argument.

For assertion (3), using Propositions \ref{conforetlaab} and Theorem
\ref{characterisation}, it suffices  to take  for $ h\in C^*(\R^{2n}) $ an
element $ c\in C^*(H_n) $, for which $ \rho(c)=h $.

The last statement follows from Proposition~\ref{intisbounded} and assertion
(3).
 \end{proof}

\begin{definition}
Let $
D_\nu(H_n) $ be the subspace of the algebra $ \F_n $, consisting of
all the fields $ (F(\la))_{\la\in\R}\in \F_n $, such that
\begin{eqnarray}
 \nn \lim_{\la\to 0}\noop{F(\la)-\nu_\la(F(0))} &=&0.
\end{eqnarray}
\end{definition}

Our main theorem of this section is the following characterisation of $C^*(H_n)$.

\begin{theorem}\label{finalres}
The Heisenberg $C^*$-algebra $C^*(H_n)$ is isomorphic to  $ D_{\nu}(H_n) $.
\end{theorem}
\begin{proof}
First we show that  $ D_\nu(H_n) $ is a $ * $-subalgebra of $\F_n$. Indeed if $
F,F'\in D_\nu(H_n) $, then
\begin{eqnarray}
\nonumber \noop{\nu_\la(F(0)+F'(0))-\pi_\la(F+F')}&\leq&
\noop{\nu_\la(F(0))-\pi_\la(F)}+\noop{\nu_\la(F'(0))-\pi_\la(F')}\to 0 \\
\nonumber &&  \text{ as } \la\to 0.
\end{eqnarray}
and since $ \lim_{\la\to0}\noop{\nu_\la(F\cdot
F'(0))-\nu_\la(F(0))\circ\nu(F'(0))}=0 $ it follows that
\begin{eqnarray}
\nonumber \noop{\nu_\la(F(0)\cdot F'(0))-\pi_\la(F\cdot F')}&\to&
0 .
\end{eqnarray}
Proposition \ref{nufhei} tells us  that $ D_\nu(H_n) $ is also
invariant under the involution $ * $.

In order to see that $ D_\nu(H_n) $ is closed, let $ F \in \F_n$ be
contained in the closure of $ D_\nu(H_n) $. Let $ \ve>0 $. Choose $
F'\in D_\nu(H_n) $, such that $ \no{F-F'}{\F_n}<\ve $. In particular,
$ \no{F(0)-F'(0)}{C^*(\R^2)}<\ve $. Thus there exists $ \la_0>0  $, such
that
\begin{equation}
 \nonumber \noop{\pi_\la(F')-\nu_\la(F'(0))}<\ve
\end{equation}
 for all $ \lav<\val{\la_0} $,
whence
\begin{eqnarray}
\nonumber \noop{\pi_\la(F)-\nu_\la(F(0))}&=&
\noop{\pi_\la(F)-\pi_\la(F') +
\pi_\la(F')-\nu_\la(F'(0))+\nu_\la(F'(0))-\nu_\la(F(0))} \\
\nonumber &\leq &  3\ve, \text{ for }\lav<\val{\la_0}.
\end{eqnarray}
Hence $ D_\nu(H_n) $ is a $ C^* $-subalgebra of $ \F_n $.

Let $ I_0 :=\{F\in \F_n, F(0)=0\}$ and let  $ I_{00}=\{F\in I_0;
\lim_{\la\to 0}\no {F(\la)}{\text{op}}=0\} $. Then $ I_0 $ and $
I_{00} $ are closed two sided ideals of $ \F_n $ and it follows from
the definition of $ \F_n $ that $ I_{00} $ is just the algebra $
C_0(\R^*,\K) $. It is clear that $ D_\nu (H_n)\cap I_0=I_{00}$. But
$ D_\nu (H_n)\cap I_0 $ is the kernel in $ D_\nu(H_n) $ of the
homomorphism $ \de_0: \F_n\to \C^*(\R^{2n}); F\mapsto F(0) $.

 Since $\im\nu\subset D_\nu(H_n)$, the canonical projection $D_\nu(H_n)\to C^*(\R^{2n}): F\mapsto
F(0) $ is surjective  and
has  the ideal $ I_{00} $ as its kernel. Thus $
D_\nu(H_n)/I_{00}=C^*(\R^{2n}) $ and  therefore $ D_\nu(H_n) $ is an
extension of $ I_{00} $ by $ C^*(\R^{2n}) $. Moreover,
\begin{eqnarray}\label{idnu}
 \nn  D_\nu(H_n)&=&I_{00}+\im \nu.
\end{eqnarray}

 Since
for every irreducible representation $ \pi $ of $ D_\nu(H_n) $, we
have either  $ \pi(I_{00})\ne0 $, and then $ \pi=\pi_\la $ for
some $ \la\in\R^* $ or $ \pi=0 $ on $ I_{00} $ and then $ \pi $ must
be a character of $ C^*(\R^{2n}) $. Hence $ \hat D_\nu(H_n)=\hat
H_n $ as sets.
That  topologies of these spaces
agree follows from the equality
\begin{equation}
 \nonumber \lim_{\la\to 0}\no{\tau(h)(\la)}{\rm op}=\no{\hat h}{\iy},\  \forall
h\in C^*(\R^{2n}),
\end{equation}
which is due to  Proposition~\ref{nufhei}.

By Theorem \ref{characterisation},  $ F(\S(H_n)) \subset D_\nu(H_n) $.
 Hence the $C^*$- algebra $
C^*(H_n) $ can be  injected into $ D_\nu(H_n) $.

Since $ D_\nu(H_n) $ is a type I algebra and the dual spaces of $
D_\nu(H_n) $ and of $ C^*(H_n) $ are the same, we have that $ F(C^*(H_n)) $ is
equal  to $
D_\nu(H_n) $ by the Stone -Weierstrass theorem (see \cite{Di}).
 \end{proof}

 \begin{remark} \rm
 Another characterisation of the $C^*$-algebra $C^*(H_n)$ is given (without proof) in a short paper  by Gorbachev \cite{gorbachev}. For $n=1$ and
 $\lambda\in\R^*$ he defines an operator-valued measure  $\mu_\la$ on $\R^2$ given on the product of two intervals $[s,t]\times[e,d]$ by
 $\mu_{\la}([s,t]\times[e,d])=P^{\frac{e}{\la},\frac{d}{\la}}FP^{s,t}F^{-1}$, where $P^{s,t}$ is the multiplication operator by the characteristic function of $[t,s]$ on  $L^2(\R)$ and $F$ is the Fourier transform  on $L^2(\R)$.
 For $f\in C_0(\R^2)$, $\la\in\R^*$ let
 $$y(f)(\la)=\int_{\R^2} f(a,b)d\mu_\la(a,b)$$
 and $y(f)(0)=f$.
  Gorbachev states that $C^*(H_1)$ is isomorphic to the $C^*$-algebra of operator fields
 $B=\{B(\la)=y(f)(\la)+a, \la\in\R^*, B(0)=f, f\in C_0(\R^2), a\in C_0(\R^*,\K)\}$.

 \end{remark}

\subsection{Almost homomorphisms and  Heisenberg property}

\begin{definition}
\rm   \rm   A  bounded mapping $ \tau:C^*(\R^{2n})\to \F_n  $ is
called an \textit{almost homomorphism } if
 \begin{eqnarray*}
 &&\lim_{\la\to 0}\noop{\tau_\la(\alpha h+\beta
f)-\alpha\tau_\la(h)-\beta\tau_\la(f)}= 0,\\
 &&\lim_{\la\to 0}\noop
{\tau_\la(h\cdot h')-\tau_\la(h)\circ \tau_\la(h')}=0, \\
&&\lim_{\la\to 0}\noop{\tau_\la (h^*)-\tau_\la(h)^*}=0, \ \al,\be\in \C, f,h\in
C^*(\R^{2n}).
\end{eqnarray*}

\end{definition}

 The mapping $\nu$ from the previous section is an example of such almost homomorphism.

 Let $\tau$ be an arbitrary almost homomorphism such that $\tau(f)(0)=f$ for any $f\in C^*(\R^{2n})$. We define as before  $
D_\tau(H_n) $ to be the subspace of the algebra $ \F_n $, consisting of
all the fields $F= (F(\la))_{\la\in\R}\in \F_n $, such that
\begin{eqnarray}
 \nn \lim_{\la\to 0}\noop{F(\la)-\tau_\la(F(0))} &=&0.
\end{eqnarray}

Using the same arguments as the one in the proof of Theorem~\ref{finalres} one can easily prove the following

\begin{proposition}\label{cheisen}
  The subspace $ D_\tau(H_n) $  of the $ C^* $-algebra $ \F_n $ is itself a $
C^* $-algebra. The algebra $ D_\tau(H_n) $ is an extension of $
C_0(\R^*,\K) $ by $ C^*(\R^{2n}) $, i.e., $C_0(\R^*,\K) $ is a
closed $*$-ideal in $ D_\tau(H_n) $ such that $D_\tau(H_n)
/C_0(\R^*,\K) $ is isomorphic to $ C^*(\R^{2n}) $.

 \end{proposition}

 \begin{definition}\label{heisenbergtype}
\rm We say that an almost homomorphism  $ \tau:C^*(\R^{2n})\to \F_n $ has
 the \textit{Heisenberg} property, if the $ C^* $-algebra  $
D_\tau(H_n) $ is isomorphic  to
 $ C^*(H_n) $.
 \end{definition}

\begin{remark}\label{wekeqhei}
\rm
As for the mapping $\nu$ we have that the dual spaces of $D_\tau(H_n)$ and of $H_n$ coincides as sets.
The necessary and the sufficient conditions for them to coincide as topological spaces is
$$\lim_{\la}\noop{\tau_\la(h)}=\|\hat h\|_{\infty},\quad h\in C^*(\R^n).$$

 \end{remark}

\begin{remark}\rm
Using the notion of Busby invariant for a $C^*$-algebra extension and the pullback algebra (\cite{wegge-olsen}), one can show that any
extension $\B\subset\F_n$ of $C_0(\R^*,\K)$ by $C^*(\R^{2n})$ is isomorphic to $D_\tau (H_n)$ for some almost homomorphism $\tau$.
The Busby invariant of such extension is $b: C^*(\R^{2n})\to C_b(\R^*,B(H))/C_0(\R^*,\K)$, $b(h)=\tau(h)+C_0(\R^*,\K)$.
\end{remark}

\textbf{Question}.

What mappings $ \tau $ give us $ C^* $-algebras $ D_\tau(H_n) $,
which are isomorphic to $ C^*(H_n) $?


\vspace{0.2cm}

Using a procedure described in \cite{delaroche} one can construct families of
$C^*$-algebras of type $D_\tau(H_n)$ which are isomorphic to $D_\nu (H_n)$ and therefore to $C^*(H_n)$.

Next example shows that there is no topological obstacle for a $C^*$-algebra of type $D_\tau(H_n)$ to be non-isomorphic to $C^*(H_n)$.
Namely, there is a $C^*$-algebras  $D_{\tau}(H_n)$ with the spectrum equal to $\widehat H_n$ and such that $D_\tau(H_n)\not\simeq C^*(H_n)$.

 We recall first that if ${\mathcal A}$, ${\mathcal C}$ are
$C^*$-algebras, then an extension of $\mathcal{C}$ by $\A$ is a
short exact sequence
\begin{equation}\label{seq}
0\rightarrow {\mathcal A}\stackrel{\alpha}\rightarrow {\mathcal B}
\stackrel{\beta}\rightarrow{\mathcal C}\rightarrow 0
\end{equation}
of $C^*$-algebras.
One says that the exact sequence splits if there is a cross-section
$*$-homomorphism $s: \mathcal C\to \B$ such that $\beta\circ s=I_{\mathcal C}$.

It is known that the extension
\begin{equation}\label{heisenseq}
0\rightarrow C_0(\R^*,\K)\rightarrow C^*(H_n)\rightarrow
C^*(\R^{2n})\rightarrow 0
\end{equation}
does not split (see \cite{rosenberg} and references therein) while
there exists a large number of splitting extensions $\B$ and
therefore non-isomorphic to $C^*(H_n)$ such that $\hat \B=\hat{H_n}$
(see \cite[VII.3.4]{delaroche}).  Here is a concrete example
inspired by \cite{delaroche}.

\begin{example}\label{example_del}\rm
Let $\{\xi_{Z}\}_{Z\in{\mathbb Z}^{2n}}$ be an orthonormal  basis of the
Hilbert space $L^2(\R^n)$. Let $P_{Z}, Z\in\Z^{2n},$ be the
orthogonal projection onto the one-dimensional $\C \xi_Z$.
  We define a homomorphism $\nu$ from $C^*({\mathbb R}^{2n})$ to $\F_n$ by
$$\nu(\va)(\lambda):=\sum _{Z\in{\mathbb Z}^{2n}}\hat \va(\val
\la^{1/2}Z)P_{Z},\la\in\R^*, \nu(\va)(0):=\va,\ \va\in C^*({\R^{2n}}).$$
We note that since for each $\la\ne 0$ and each compact subset $K\subset
\R^{2n}$,
the set $\{Z\in\Z^{2n}:\val\la^{1/2}\in K\}$ is finite and since $\hat\va\in
C_0(\R^{2n})$, one can
easily see that $\nu(\va)(\la)$ is compact.
Moreover
\begin{equation}
\nonumber \noop{\nu(\va)(\lambda)}=
\sup_{  Z\in \Z^{2n}}\val
{\va(\val \la^{1/2} Z)}.
\end{equation}

Since we can find for   every vector $ u\in\R^{2n} $ and $ \la\in\R^* $ a
vector $ Z_\la\in \Z^{2n} $, such that $ \lim_{\la\to 0} \lav^{1/2} Z_\la=u $,
we
see that
\begin{eqnarray}\label{slitdelaro}
 \lim_{\la\to 0}
\noop{\nu_\la(\va)}=\no{\hat\va}\iy=\no
\va{\C^*(\R^{2n})} .
\end{eqnarray}

\end{example}

\section{The $C^*$-algebra of the thread-like Lie groups  $ G_N$ }\label{fad}

 For $N \ge {3}$, let ${\frak g}_N$ be the
$N$-dimensional real nilpotent Lie algebra with basis
$X_{1},\ldots,X_N$ and non-trivial Lie brackets
$$
[X_N,X_{N-{1}}] =X_{N-{2}}, \ldots,[X_N,X_{2}] =X_{1}.
$$
The Lie algebra ${\frak g}_N$ is $(N-{1})$-step nilpotent and is a semi-direct
product of $\Re X_N$ with the abelian ideal
\begin{equation}\label{idealing}\b:=\sum_{j={1}}^{N-{1}} \R
X_j.\end{equation} Let
  $$\b_j:=\text{span}\{X_i,i=1,\cdots,j\}, 1\leq j\leq N-1.$$

Note that ${\frak g}_{3}$ is the three dimensional Heisenberg Lie algebra. Let
$G_N:=\exp{\gg_N}$ be the associated connected, simply connected
Lie group. Let also $B_j:=\exp{\b_j}$ and $B:=\exp{\b}$.
 Then for $3\leq M\leq N$ we have $ G_M\simeq G_N/B_{N-M}$.

\subsection{The unitary dual of $G_N$}
In this section we describe the unitary irreducible representations of $G_N$ up
to a unitary equivalence.

For $\xi =\sum_{j={1}}^{N-{1}} \xi_j X_j^* \in {\frak g}_N^*$, the
coadjoint action is given by
$$
\Ad^*(\exp{-tX_N})\xi =\sum_{j={1}}^{N-{1}} p_j(\xi,t) X_j^*,
$$
where, for ${1} \le j \le N-{1}$, $p_j(\xi,t)$ is a polynomial in
$t$ defined
 by
$$
p_j(\xi,t) =\sum_{k=0}^{j-{1}} \frac{t^k}{k!} \xi_{j-k}.
$$
Moreover, if $\xi_j \neq 0$ for at least one ${1} \le j \le N-{2}$,
then $\Ad^*(G_N)\xi$ is of dimension two, and $\Ad^*(G_N)\xi
=\{\Ad^*(\exp{t X_N})\xi + \R X_N^*, t\in\R\}$.  We shall
always identify ${\mathfrak g}_N^*$ with $\R^N$ via the mapping
$(\xi_{N},\ldots,\xi_1) \to
\sum_{j={1}}^N \xi_j X_j^*$
and  the subspace  $V =\{ \xi \in {\mathfrak g}_N^*: \xi_N
=0 \}$ with the dual space of $ \b $. For $\xi \in V$ and
 $t
\in \R$, let
$$
t\cdot \xi =\Ad^*(\exp{t  X_N})\xi$$
\begin{equation}\label{action of R}
=\Big(0,   \xi_{N-{1}} - t \xi_{N-{2}} + \ldots +
\frac{{1}}{(N-{2})!} (-t)^{N-{2}} \xi_{1},\ldots,\xi_{2} - t\xi_{1},
\xi_{1}\Big).
 \end{equation}
As in \cite{A.K.L.S.S.}, we define the function $\widehat{\xi}$ on
$\Re$ by
\begin{equation}\label{polynomialfunction}
 \widehat{\xi}(t):=(t\cdot\xi)_{N-1}
=\xi_{N-{1}} - t \xi_{N-{2}} + \ldots + \frac{{1}}{(N-{2})!}
(-t)^{N-{2}} \xi_{1}.
 \end{equation}
Then the mapping $\xi \to \widehat{\xi}$ is a linear isomorphism
of $V$ onto $P_{N-{2}}$, the space of real polynomials of degree
at most $N-{2}$. In particular, $\xi_k \to \xi$ coordinate-wise in
$V$ as $k \to \infty$ if and only if $\widehat{\xi}_k( t) \to
\widehat{\xi}(t)$ for all $t \in \Re$. Also, the mapping $\xi \to
\widehat{\xi}$ intertwines the $\Ad^*$-action and translation in
the following way:
$$
\widehat{t \cdot \xi}(s) =(s \cdot (t \cdot \xi))_{N-{1}}
$$
$$ =((s+t)\cdot \xi)_{N-{1}} =\widehat{\xi}(s+t)
$$
for $\xi \in V$ and $s, t \in \Re$. \me

By Kirillov's orbit picture of the dual space of a nilpotent Lie
group, we can describe the  irreducible unitary representations of
$G_N$ in the following way (see \cite{CG} for details). For any
non-constant polynomial $p=\hat \ell \in P_{N-{2}}$ we consider
the induced representation $\pi_\ell=\text{ind}_B^G\ch_\ell$,
where $\chi_\ell$ denotes the unitary character of the abelian
group $B$ defined by:
 \begin{equation*}
  \chi_\ell(\exp U)=e^{- 2\pi i \sp \ell U }, U\in \b.
  \end{equation*}
Since $\b$ is abelian of codimension ${1}$, it is a polarization
at $\ell$ and so $\pi_\ell$ is irreducible. Every infinite
dimensional irreducible unitary representation of $G_N$ arises in
this manner
up to equivalence.   \\

Let us describe the representation $ \pi_\ell,\ell\in \b^*, $
explicitly. The Hilbert space $ \H_\ell $ of the representation $
\pi_\ell$ is the space $ L^2(G_N/B,\ch_\ell) $ consisting of all
measurable functions $ \tilde\xi:G_N\to \C$, such that $
\tilde\xi(gb)=\ch_\ell(b\inv)\tilde\xi(g) $ for all $ b\in B $ and
all $ g\in G  $ outside some set of measure of Lebesgue measure 0
and such that the function $ \val{\tilde\xi} $ is contained in $
L^2(G_N/B)$. We can identify the space $ L^2(G_N/B,\ch_\ell) $ in an
obvious way with $ L^2(\R) $ via the isomorphism $ U:\xi\mapsto
\tilde\xi$ where $\tilde\xi(\exp{ s X_N}b):= \ch_\ell(b\inv)\et(s),
s\in\R,b\in B $. Hence for $g=\exp{t X_N}b$ and $\xi\in L^2(\R)$ we
have an explicit expression for the operator $
\pi_\ell(g) $:
\begin{eqnarray}\label{pilexpressed}
 \pi_\ell(g)\xi(s)&=& \tilde \xi(g\inv \exp{ s X_N}) \\
\nonumber &=&  \tilde \xi(b\inv \exp{(s-t)X_N})\\
\nonumber &=& \tilde \xi(\exp{(s-t)X_N} (\exp{(t-s)X_N}b\inv \exp{(s-t)X_N})) \\
\nonumber &=&  \ch_\ell(\exp{(t-s)X_N}b \exp{(s-t)X_N})\xi(s-t)\\
\nn  &=&e^{{-2\pi i}  \ell(\rm{Ad}(\rm{exp}((t-s)X_N)\log(b))}\xi(s-t), s\in\R.
\end{eqnarray}

We can parameterize the orbit space $ \g_N^*/G_N  $ in the
following way. First we have a decomposition
\begin{eqnarray}
 \nn  \g_N^*/G_N=\bigcup_{j=1}^{N-2}{\g_N^*}^j/G_N \bigcup X^*,
\end{eqnarray}
where
\begin{eqnarray}
 \nn {\g_N^*}^j:=\{\ell\in\g_N^*, {\ell(X_i)=0, i=1,\cdots,j-1, \ell(X_j)\ne
0}\}
\end{eqnarray}
and where
\begin{eqnarray}
 \nn  X^*:=\{\ell\in {\g^*_N},
\ell(X_j)=0, j=1,\cdots,N-2\}
\end{eqnarray}
denotes the characters of $ G_N $.
A character of the group $ G_N $ can be written as $ \ch_{a,b}, a,b\in\R, $
where\begin{eqnarray}
 \nn \ch_{a,b}(x_N,x_{N-1},,\cdots, x_1):=e^{{-2\pi i}   ax_N {-2\pi i}   b
x_{N-1}},\
(x_N,,\cdots, x_1)\in G_N.
\end{eqnarray}

For any $ \ell\in   {\g_N^*}^j, N-2\geq j\geq 1$ there exists
exactly one element $ \ell_0 $ in the $ G_N $-orbit of $ \ell $,
which satisfies the conditions
$$\ell_0(X_j)\ne0, \ell_0(X_{j+1})=0, \ell_0(X_N)=0.$$

We can thus parameterize the orbit space $\g^*_N/G_N $, and hence
also the dual space $ \widehat {G_N} $, with the
sets\begin{eqnarray}
 \nn S_N &:=&\bigcup_{j=1}^{N-2}S_N^j\bigcup X^*,
\end{eqnarray}
where $ {S_N^j}:= \S_N\cap {\g^*}^j=\{\ell\in
{\g^*_N}^j,\ell(X_{k})=0, k=1,\cdots, j-1, j+1, \ell(X_j)\ne 0 \}
$.   Let
\begin{eqnarray}
 \nn S_N^{gen} &:=&\bigcup_{j=1}^{N-2}S_N^j
\end{eqnarray}
 be the family of points in $ S_N $, whose $ G_N $-orbits are of dimension $ 2
$.

\subsection{The topology of $\widehat{G_N}$}

The topology of the dual space of $ G_N $ has been studied in detail
in the papers \cite{A.L.S.} and \cite{A.K.L.S.S.} based on the methods developed
in \cite{LRS} and \cite{L}. We need the
following description of the convergence of sequences $(\pi_k)_k$ of
representations in $\widehat{G_N}$.

Let $(\pi_k)_k$ be a sequence in $\widehat{ G_N}$. It is said to be
{\it properly convergent} if it is convergent and all cluster points
are limits. It is known (see \cite{LRS}) that any convergent
sequence has a properly convergent subsequence.

\begin{proposition}\label{sequences}

Suppose that
$(\pi_k=\pi_{\ell_k})_k, (\ell_k\in S_N^{\rm gen}, k\in\N)$ is a sequence in
$\widehat{G_N}$ that
 has a cluster point. Then
there exists a subsequence, (also indexed by the symbol $ k $ for
simplicity), called   \textit{with perfect data}   such that
$(\pi_k)_k$ is properly converging and  such that the polynomials $
p_k$, $k\in\N,$  associated to  $\pi_k$  have the following
properties: The polynomials $ p_k $ have all the same degree $ d $.
Write  $$p_k(t):= c_k \prod_{j=1}^d (t-a_j^k)=\hat \ell_k(t),t\in\R,
\ell_k\in V.$$

 There
exist  $ 0<m\leq 2d $, real sequences
 $ (t_i^ k)_k$ and polynomials $q_i$ of degree $d_i\leq d$, $i=1,\cdots, m$,
such that
\begin{enumerate}
\item  $\lim _{k\to\iy}p_k(t+t_i^k)\to q_i(t)$, $t\in \R$, $1\leq
i\leq m$ or equivalently
 $ \lim_{k\to\iy} t_i^k\cdot \ell_k\to \ell^i$,  where $ \ell^i $ in $ V $ such
that $\hat\ell^i(t)=q_i(t)$.
 \item  $\lim_{k\to\iy}
\val{t^k_i-t^k_{i'}}=+\iy$, for all $ i\ne i'\in\{1,\cdots,m\} $.
\item  If $ C=\{i\in \{1,\cdots,m,\}, \ell^i  \text{ is a
character }\} $ then for all $i\in C$
\begin{enumerate}
\item $\lim_{k\to\iy}\val{t_i^k-a^k_j}=+\iy$ for all $
j\in\{1,\cdots,d\} $;
 \item  there exists an index $ j(i)\in
\{1,\cdots,d\} $ such that $ \val{t^k_{i}-a^k_{j(i)}}\leq
\val{t^k_{i}-a^k_j} $  for all $ j\in \{1,\cdots,d\}$; let
$$\rh^k_i:= \val{t^k_{i}-a^k_{j(i)}};$$
\item there exists a subset $ L(i)\subset \{1,\cdots,m\}$, such
that $\lim_{k\to \iy}
\frac{\val{t^k_{i}-a^k_j}}{\rh^k_i} $ exists in $
\R $ for every $ j\in L(i) $ and such that $\lim_{k\to \iy}
\frac{\val{t^k_{i}-a^k_j}}{\rh^k_i}=+\iy $ for $
j\not \in L(i) $;

\item  the polynomials $ (t^k_i+s \rh_i^ k)\cdot p_k  $ in $ t $
converge uniformly on compacta to the constants
$$\lim_{k\to\iy} (t^k_i+s \rh_i^ k)\cdot p_k(t) =p^i(s), s\in\R, $$
 and these constants define a real polynomial of degree $ \# L(i) $ in $ s $.
\item  If $ i'\ne i\in C $, then $ L(i)\cap L(i')=\es $.
 \end{enumerate}

\item  Let  $ D=\{1,\cdots,m\}\setminus C $ and  write $
\rh_i^k:=1 $ for $ i\in D $. For $ i\in D $, let $$ J(i):=\{1\leq
j\leq d, \lim_{k\to\iy}\vert t^k_i-a^k_j\vert=\iy\} .$$

  Suppose that
$ (t_k)_k $ is a real sequence, such that  $ \lim_{k\to\iy}
t_k\cdot \ell_k\to \ell$ in $ \g_N^* $, then
\begin{enumerate}
\item if $\ell$ is a non-character, then the sequence
$(|t_k-t_i^k|)_k$ is bounded for some $i\notin C$; \item if $\ell$
is a character, then
$\lim_{k\to\infty}\left|\frac{t_k-a_j^k}{t^k_i-a_j^k}\right|$
exists for some $i\in C$ and some $j\in L(i)$ and $
\ell\res\b=q_i(s) X^*_{N-1}$ for some $ s\in\R $.
\end{enumerate}

\item  Take any real sequence  $ (s_k)_k  $, such that $
\lim_{k\to\iy }\vert s_k \vert =+\iy $, and such that for any $ i\in
D $, $j\in J(i)$, $\frac{s_k}{|t_i^k-a_j^k|}\to 0$, and  for $ i\in
C $, $j\notin L(i)$, $\frac{s_k\rho_i^k}{|a_j^k-t_i^k|}\to 0$ and
$\frac{s_k}{\rho_i^k}\to 0$ as $k\to\infty$. Let
$$S_k:=
(\bigcup_{i=1}^m [t_i^k-s_k\rh^k_i,t_i^k+s_k\rh^k_i]); \
T_k:=\R\setminus S_k, k\in \N.$$ Then for any sequence $(t_k)_k$,
$t_k\in T_k,$ we have $t_k\cdot l_k\to\infty$.

We say that the sequence $(s_k)_k$   is \textit{ adapted to the
sequence $ (\ell_k) $. }
\end{enumerate}
 \end{proposition}

\begin{proof}
We may assume that $(\pi_k)_k$ is properly convergent with limit set $L$.  We
can also assume, by passing to a subsequence, that
 each $p_k$ has degree $d$.  By \cite{L} the number of non-characters in $L$ is
finite. Let  this subset of non-characters be denoted by $L^{gen}$.
If $L^{gen}$ is non-empty by passing further to a subsequence we may assume the
sequence $(\pi_k)_k$ converges $i_{\sigma}$-times to each character $\sigma\in
L^{gen}$ (see p.34, \cite{A.K.L.S.S.} for the definition of $m$-convergence and
p.253 \cite{A.L.S.}). Let $s=\sum_{\sigma\in L^{gen}}i_{\sigma}$. Then there
exist
 non-constant polynomials $q_1,\ldots,q_s$
of degree $d_i\leq d$, $i=1,\ldots,s$, and sequences
$(t_1^k)_k,\ldots, (t_s^k)_k$ such that the conditions (1) and (2)
are fulfilled and for each $\sigma\in L^{gen}$ there are
$i_{\sigma}$ equal polynomials amongst $q_1,\ldots, q_s$
corresponding to $\sigma$. Then if $(t_k)_k$ is a real sequence such
that $t_k\cdot \ell_k\to \ell$, $\ell$ is a non-character, then
$\ell$ corresponds to some $\sigma\in L^{gen}$ and we may assume
that $\hat \ell=q_i$ for some $i\in{1,\ldots,s}$. It follows from
the definition of $i_{\sigma}$-convergence  that the sequences
$(t_k\cdot\ell_k)$ and $(t_i^k\cdot\ell_k)$ are not disjoint
implying $|t_k-t_i^k|$ is bounded and therefore (4a).

If $(\pi_k)_k$ has a  character as a limit point then passing if necessary to a
subsequence  we can find a maximal family of real sequences $(t_l^k)_k$,
$l<s\leq m\leq d$, constant polynomials $q_l$, non-negative sequences
$(\rho_l^k)_k$ and
polynomials $p_l$ satisfying $(1)-(4)$ (see Definition~6.4 and the discussion
before in \cite{A.L.S.}).

The condition (5b) follows from the maximality of the family of sequences
$(t_l^k)_k$ and the proof of Proposition~6.2, \cite{A.L.S.}.

Suppose now that we have a sequence $ (t_k)_k $ such that $ t_k\in
T_k $ for every $ k $ and such that some subsequence (also indexed
by $ k $ for simplicity of notations)  $ ( t_k\cdot \ell_k)_k $
converges to an $ \ell \in \g^*$. By condition (5) then either for
some $ i\in D $, the sequence $ (t_k-t^k_i)_k $ is bounded, i.e $
t_k\in [t_i^k-s_k\rh_i^k,t_i^k+s_k \rh_i^k]  $ for $ k $ large
enough, which is impossible, or we have an $ i\in C $, such that
$\lim_{k\to\infty}\left|\frac{t_k-a_j^k}{t^k_i-a_j^k}\right|$ exists
for some $ j\in L(i) $. But then
\begin{eqnarray}
\nonumber \frac{\val{t_k-t_i^k}}{\rh_i^k}&\leq&\frac{\val{t_k-a_j^k}}{\rh_i^k} +
\frac{\val{t_i^k-a_j^k}}{\rh_i^k}
\\
\nonumber &=&
\frac{\val{t_k-a_j^k}}{\val{t_i^k-a_j^k}}\frac{\val{t_i^k-a_j^k}}{\rh_i^k}
+\frac{\val{t_i^k-a_j^k}}{\rh_i^k}
\end{eqnarray}
and so the sequence $ (\frac{\val{t_k-t_i^k}}{\rh_i^k})_k $ is bounded, i.e
$ t_k\in  [t_i^k-s_k\rh_i^k,t_i^k+s_k
\rh_i^k]\subset S_k  $ for $ k $ large enough, a contradiction. Hence
 $ \lim_k t_k\cdot \ell_k=\iy $ whenever $ t_k\in T_k $ for large $ k
$.

\end{proof}


\begin{example} \label{ex_heis}
\rm
Let us consider the Heisenberg group $G_3$. Then $S_3=S_3^1\cup \R^2$. Let $(\ell_k)\in S_3^1$.  Then $\ell_k=\la_k X_1^*$,
$\la_k\in \R^*$.
The associated polynomials are $p_k(t)=-\la_k t$ ($d=1$, $c_k=-\la_k$, $a_1^k=1$). Assume that $(\ell_k)$ is a sequence with perfect data. Then either $\pi_{\ell_k}$ converges to $\pi_{\ell}$, $\ell\in S_3^1$, or
$\pi_{\ell_k}$ converges to a character and in this case $\la_k\to 0$ as $k\to\infty$. We shall consider now the second case.
So we have $m=1$ and $\ell^1=X_2^*$ with the corresponding polynomial $q_1(t)=1$ and $t_1^k=-1/\la_k$ and thus $\rho^k_1=1/|\la_k|$.
The polynomial $p^1(s)$ is the limit $$\lim_{k\to\infty}p_k(t_1^k+s\rho_i^k+t)=\lim_{k\to\infty}(-\la_k)(-1/\la_k+s/|\la_k|+t)=\lim_{k\to\infty}(1-\text{sign}\la_k s-\la_k t).$$
Since $(\ell_k)$ is a sequence with perfect data, the sign of $\la_k$ is constant, implying
$q_11(s)=1+\epsilon s$, where $\epsilon=\pm 1$. A real sequence $(s_k)$ is
adapted to $(\ell_k)$ if and only if $s_k\to\infty$ and $s_k|\la_k|\to 0$.

\end{example}

\subsection{A $C^*$-condition}

Let $C^*(G_N)$ be the full $C^*$-algebra of $G_N$ that is the completion of the convolution algebra $L^1(G_N)$ with respect to the norm
$$\|f\|_{C^*(G_N)}=\sup_{\ell\in\S_N}\noop{\int_{G_N} f(g)\pi_{\ell}(g) dg}.$$

\begin{definition}\label{fourztrcompsupp}
\rm Let $ f\in L^1(G_N) $. Define the function $  \hat f^2 $ on $
\R\times \b^ *$ by $$ \hat f^2(s,\ell):=\int_{B} f(s, u)e^{{-2\pi i}
\ell(\log(u))} du, s\in\R, \ell\in\b^*.$$

We denote by $ L^1_c (G_N)$ the space of functions $ f\in L^1(G_N)
$, for which $ \hat f^2 $ is contained in $ C_c^\iy(\R\times\b^*) $,
the space of compactly supported $C^{\infty}$-functions on $\R\times
\b^*$. The subspace $ L^1_c (G_N)$ is dense in $ L^1(G_N) $ and
hence in the full $C^*$-algebra $ C^*(G_N)$
 of $G_N$.
 \end{definition}

\begin{proposition}\label{cstindst}
Take $ f\in L^1_c (G_N)$ and let $ \ell\in S_N^{gen} $. Then the
operator $ \pi_\ell(f) $ is a kernel operator with kernel function
$$f_\ell(s,t)=\hat f^2(s-t, t\cdot \ell), s,t\in\R. $$
 \end{proposition}
\begin{proof}
Indeed, for $\xi\in L^2(\R), s\in\R,$ we have that
\begin{eqnarray}\label{star}
\nonumber \pi_\ell(f)\xi(s)&=& \int_{G_N}f(g)\pi_\ell(g)\xi dg \\
\nonumber &=& \int_\R\int_B f(t,b)
e^{{-2\pi i}  \ell(\rm{Ad}(\rm{exp}((t-s)X_N)\log(b))}\xi(s-t)db dt (\text{ by
\ref{pilexpressed}})\\
 &=& \int_\R\int_B f(s-t,b)
e^{{-2\pi i}  {(\rm Ad}^*({\rm exp}({tX_N})(\ell)(\log(b))}\xi(t)db dt\\
\nonumber &=& \int_\R \hat f^2(s-t,{\rm Ad}^*({\rm exp}({tX_N}(\ell))
\xi(t)dt\\
\nonumber &=& \int_\R\hat f^2(s-t,t\cdot \ell)\xi(t) dt.
\end{eqnarray}
\end{proof}
\begin{definition}\label{leftgoverc}
\rm  Let $ \c:=\text{ span }\{X_1,,\cdots, X_{N-2}\} $. Then $ \c $ is an
abelian ideal of $ {\g_N} $, the algebra $ \g_N/\c $ is abelian and isomorphic
to $ \R^2 $ and $ C:=\exp{\c} $ is an abelian closed normal subgroup of $ G_
N.$

Let
$$ \rho=\ind C{G_N }1 $$ be the left regular representation of $
G_N $ on the Hilbert space $ L^2(G_N/C) $. Then the image $
\rho(C^*(G_N)) $ is  the $ C^*$-algebra of $ \R^2 $ considered
as algebra of convolution operators on $ L^2(\R^2) $ and hence $ \rho
(C^*(G_N))$ is isomorphic to the algebra $ C_0(\R^2) $ of continuous
functions vanishing at infinity on $ \R^2 $.
As for the Heisenberg algebra we have that if $f\in L^1(G_N)$ then the Fourier transform
$\widehat{\rho(f)}(a,b)$ of $\rho(f)\in C^*(\R^2)$ equals $\hat f(a,b,0,\ldots,0)$.
 \end{definition}

Our aim is to realize the $C^*$-algebra $C^*(G_N)$ as a
$C^*$-algebra of operator fields.
\begin{definition}\label{fourN}
\rm   For $ a\in C^*(G_N) $ we define the Fourier transform $F(a)$ of $
a
$ as operator field
$$ F(a):=\{(A(\ell):=\pi_\ell(a), {\ell\in S_N^{gen}},
A(0):=\rho(  a)\in C^*(\R^2)\}. $$

 \end{definition}

\begin{remark}\label{continorm}
\rm   We observe that the spaces $ S_N^j, j=1,\cdots, N-2,  $ are Hausdorff
spaces if we equip them with the topology of $ \widehat {G_N} $. Indeed, let $
(\ell_k)_k $ be  a sequence in $ S_N^j $, such that the sequence of
representations $ ({\pi_{\ell_k}})_{k }$ converges to some $ \pi_\ell $ with $
\ell\in S_N^j$. Then the numerical sequence $ (\la_k:=\ell_k(X_k))_k $ converges
to $ \la:=\ell(X_j)\ne 0 $. Suppose now that the same sequence $ (\pi_{\ell_k})
$ converges to some other point  $\pi_{\ell'} $. Then there exists a numerical
sequence $ (t_k)_k $ such that $ {\Ad^*(\exp{t_k X_N})\ell_k}\res{\b} $
converges to $ \ell'\res\b $. In particular $-\la_k t_k= \Ad^*(\exp{t_k
X_N})\ell_k(X_{j+1} )\stackrel{k\to\iy}{\to} \ell'(X_{j+1})$. Hence the
sequence $ (t_k)_k $ converges to some $ t\in \R $ and $
\pi_{\ell'}=\pi_\ell$.  Similarly, we see from (\ref{star}) that
for $
f\in L^1_c(G_N) $, the
mapping $ \ell\to \pi_\ell(f) $ is norm continuous when restricted
to  the sets $ S_N^j, j=1,\cdots, N-2 $, since for the sequence $
(\pi_{\ell_k})_k $ above, the functions $ f_{\ell_k} $ converge in
the $L_2$-norm to $ f_\ell $.
 \end{remark}

\begin{definition}\label{condpq}
{\rm

Define for $ t,s\in \R $ the selfadjoint projection operator on $
\l2 \R$ given by
$$M_{t,s}\xi(x):=1_{(t-s,t+s)}(x)\xi(x), x\in \R, \xi\in\lt,$$
where $ 1_{(a,b)}, a,b\in\R, $ denotes the characteristic function
of the interval $ (a,b)\subset \R.  $

We put for $ s\in\R $
$$M_s:=M_{0,s}.$$

 More generally, for a measurable subset $ T\subset\R $, we
let $ M_T $ be the multiplication operator with the characteristic
function of the set $ T $. For $ r\in\R $, let $ U(r) $ be the
unitary operator on $ \l2\R $ defined by
\begin{eqnarray}
\nn U(r)\xi(s) &:=&\xi(s+r),\xi\in\l2\R,s\in\R.
\end{eqnarray}
}
\end{definition}

\begin{definition}
\rm Let $ (\pi_{\ell_k})_k $ be a properly converging sequence in
$ \widehat G_N $ with perfect data $ ((t^k_i)_k, (\rh_i^k), (s^k_i))
$.
 Let $ i\in C $ and let $ \et\in \D(\R^n) $ such that $ \et$ has $ L^2 $-norm
1. Define for $ \rh^k_i,k\in\N, i\in C, $ and $ u=(a,b)\in\R^{2} $
the Schwartz function
\begin{eqnarray}
 \nn \et(i,k,u)(s) &:=&\et(s_k p^i\left(\frac{s}
{\rh^k_i}\right)+s_k b)e^{{2\pi i a}\cdot s}, s\in\R.
\end{eqnarray}

By Example~\ref{ex_heis}, for $N=3$ we have
$$\et(1,k,u)=\et (\pm s_k|\la_k|s+s_k(1+b))e^{{2\pi i a}\cdot s}.$$

Let $ P_{i,k,u} $ be the operator of rank one defined by
\begin{eqnarray}\label{defpiku}
 \nn P_{i,k,u} \xi&:=&\langle{\xi},{\et(i,k,u)}\rangle\et(i,k,u), \xi\in \lt.
\end{eqnarray}
\end{definition}

\begin{definition}\label{pnuk}
\rm  For an  element $ \va\in \S(\R^2) $ let \begin{eqnarray}
\nn \nu(\va)(i,k) &:=s_k\int_{\R^2} \hat  \va(a,-b)P_{i,k,u}{dadb}, k\in\N,
i\in C.
\end{eqnarray}
 \end{definition}

Then for $ \va\in \S(\R^2),\xi\in L^2(\R), s\in\R, $ we have that
\begin{eqnarray}\label{diffnuk}
 \nn \nu(\va)(i,k)(\xi)(s) &:=&s_k \int_{\R^2} \hat
\va(a, -b)(P_{i,k,u}\xi)(s){du} \\
\nonumber &=&  s_k\int_{\R^{2}}\hat \va(a,-b)\left( \int_{\R}\xi(t)
\overline{\et(s_k p^i\left(\frac{t} {\rh^k_i}\right)+s_k b)}e^{{-2\pi i}
a\cdot (t-s)} dt \right)\\
\nonumber&&\et(s_k p^i\left(\frac{s}
{\rh^k_i}\right)+s_k b)db da \\
\nonumber &=& s_k \int_{\R}\int_{\R}\hat \va^2(s-t,-b) \xi(t)
\overline{\et(s_k p^i\left(\frac{t} {\rh^k_i}\right)+s_k b)} \et(s_k
p^i\left(\frac{s}
{\rh^k_i}\right)+s_k b)dtdb  \\
&=& \int_{\R}\int_{\R}\hat \va^2(s-t,-\frac{b}{s_k}+p^i\left(\frac{t}
{\rh^k_i}\right))
 \overline{\et(b)}\\
 \nonumber && \et(s_k\left(p^i\left(\frac{s}{\rh^k_i}\right)
-p^i\left(\frac{t} {\rh^k_i}\right)\right)+b)\xi(t) dtdb.
\end{eqnarray}

Since $\et$ has $L_2$-norm $1$, using (\ref{star}) and
(\ref{diffnuk})
 we get
\begin{eqnarray}\label{seqik}
\nn && (U(t^k_i)\circ \pi_{\ell_k}( f)\circ U(-t^k_i)\circ
M_{s_k}-\nu (F(f)(0))(i,k)\circ
M_{s_k}) (\xi)(s)\\
\nn &=&\int_{-s_k}^{s_k}(\int_{\R}\hat f^2(s-t,(t+t^k_i)\cdot
\ell_k))-\hat f^2(s-t,-\frac{b}{s_k}+p^i\left(\frac{t}
{\rh^k_i}\right),0\ldots)\\
&&
 \overline{\et(b)} \et(s_k\left(p^i(\frac{s}{\rh^k_i})
-p^i\left(\frac{t} {\rh^k_i}\right)\right)+b)db)\xi(t) dt\\
&&\nn +\int_{-s_k}^{s_k}(\int_{\R}\hat f^2(s-t,(t+t^k_i)\cdot
\ell_k))
 \overline{\et(b)} \\
 \nn &&(\eta(b)-\et(s_k\left(p^i\left(\frac{s}{\rh^k_i}\right)
-p^i\left(\frac{t} {\rh^k_i}\right)\right)+b))db)\xi(t) dt.
\end{eqnarray}

\begin{proposition}\label{pikjcom}
Let $ \va\in C^*(\R^2) $,  $i\in C$ and $ k\in \N $. Then
\begin{enumerate}
\item the operator $ \nu(\va)(i,k) $ is
compact and $ \noop{\nu(\va)(i,k) }\leq \no\va{C^*(\R^2)}$;
\item  we have that $ \nu(\va)(i,k)^*=\nu(\va^*)(i,k)  $;
\item  furthermore $$ \lim_{k\to\iy}\noop{\nu(\va)(i,k)\circ
(\Id-M_{s_k\rho_i^k})}=0 $$
and hence
$$\lim_{k\to\iy}\noop{(\Id-M_{s_k\rho_i^k})\circ\nu(\va)(i,k)\circ
M_{s_k\rho_i^k}}=0.$$
 \end{enumerate}

 \end{proposition}
\begin{proof}
1.) It suffices to prove this for $ \va\in \D(\R^2) $. We have that

\begin{eqnarray}
\nonumber \no{\nu(\va)(i,k) \xi}2^2&=& \int_{\R}\vert
\int_\R\int_{\R}\hat \va^2(s-t,-\frac{b}{s_k}) \xi(t)
\overline{\et(s_k p^i(\frac{t} {\rh^k_i})+ b)}dt \et(s_k
p^i(\frac{s}
{\rh^k_i})+ b)db \vert^2ds \\
\nn &=&\int_{\R}\vert \int_\R(\hat \va^2(-,-\frac{b}{s_k})\ast (\xi
\overline{\et_{k,b}}))(s) \et(s_k p^i(\frac{s}
{\rh^k_i})+ b)db \vert^2ds \\
\nn&&(\text{where } \et_{k,b}(t):=\et(s_k p^i(\frac{t}
{\rh^k_i})+ b), t\in\R)\\
\nonumber &\leq& \int_{\R^2}\vert (\hat \va^2(-,-\frac{b}{s_k})\ast
(\xi \overline{\et_{k,b}}))(s)\vert^2
db ds \\
\nonumber &\leq&\no{\va}{C^*(\R^2)}^2\int_{\R}\Vert\xi
\et_{k,b}\Vert_2^2
{db}\\
\nonumber &=&  \no{\va}{C^*(\R^2)}^2\int_{\R^2}\vert \xi(t)\vert^2
\vert \et(s_kp^i(\frac{t}{\rh^k_i})+b)\vert^2
{db}dt\\
\nonumber &=&  \no{\va}{C^*(\R^2)}^2\no\xi2^2.
\end{eqnarray}
 Furthermore, since $ \nu(\va)(i,k) $ is an integral of rank one operators,  $
\nu(\va)(i,k) $ must be compact. Hence for every $ \va\in
C^*(\R^2) $, $\nu(\va)(i,k)$  is  a compact operator bounded by $
 \no \va{C^*(\R^2)}$.

2.) Let $ \va\in \S(\R^2) $. Then $ \ol{\hat \va}=\hat {\va^*} $ and
so\begin{eqnarray}
 \nn \nu(\va)(i,k)^*&=&(s_k\int_{\R^2} \hat \va(u)P_{i,k,u}{du})^*=
s_k\int_{\R^2} \ol{\hat \va(u)}P_{i,k,u}{du}\\
\nn  &=&s_k\int_{\R^2} \hat {\va^*}(u)P_{i,k,u}{du}=\nu(\va^*)(i,k).
\end{eqnarray}
3.) Take now $ \va\in \S(\R^2) $, such that $ \hat\va$ has a compact
support. We denote by $[-s_k\rho_i^k,s_k\rho_i^k]^c$ the set
$\R\setminus [-s_k\rho_i^k,s_k\rho_i^k]$. By (\ref{diffnuk})  for
any $ \xi\in L^2(\R), s\in\R $ we have

\begin{eqnarray}
\nn  && \nn \nu(\va)(i,k)\circ (\Id-M_{s_k\rho_i^k})(\xi)(s)\\
\nn &=&\int_{ [-s_k\rho_i^k,s_k\rho_i^k]^c}\int_{\R}\hat
\va^2(s-t,-\frac{b}{s_k}+p^i(\frac{t} {\rh^k_i}))
 \overline{\et(b)} \et(s_k(p^i(\frac{s}{\rh^k_i})
-p^i(\frac{t} {\rh^k_i}))+b)db\xi(t) dt=0
\end{eqnarray}
since for $ k $ large enough $ \hat
\va^2(s-t,-\frac{b}{s_k}+p^i(\frac{t} {\rh^k_i}))=0 $ for any $
t\in [-s_k\rho_i^k,s_k\rho_i^k]^c$, $b\in \supp \et $, $s\in \R$. Hence $
\nu(\va(i,k))\circ (\Id-M_{s_k\rho_i^k})=0 $ for $ k $ large enough.
Since the mapping $ \nu $ is continuous, it follows that $
\lim_{k\to\iy}\noop{\nu(\va)(i,k)(\Id-M_{s_k\rho_i^k})}=0$ for all
$ \va\in C^*(\R^2)$ and every $ i\in C $.

Hence also
\begin{eqnarray}
 \nn  &&\lim_{k\to\iy}\noop{(\Id-M_{s_k\rho_i^k})\circ\nu(\va)(i,k)\circ
M_{s_k\rho_i^k}}\\
\nn  &=&\lim_{k\to\iy}\noop{(M_{s_k\rho_i^k}\circ\nu(\va^*)(i,k)\circ
(\Id-M_{s_k\rho_i^k})}\\
\nn  &\leq &\lim_{k\to\iy}\noop{\nu(\va^*)(i,k)\circ
(\Id-M_{s_k\rho_i^k})} =0.
\end{eqnarray}
\end{proof}

\begin{definition}\label{gencon}
\rm   Let Let $ A=(A(\ell)\in\K(\l2\R),\ell \in S_N^{gen}, A(0)\in
C^*(\R^2)) $ be a field of bounded
operators. We say that $ A $ satisfies the \textit{generic
condition} if for every properly converging sequence  with perfect
data $ (\pi_{\ell_k})_k\subset \hat G_N $ and for every limit point
$ \pi_{\ell^i}, i\in D,
 $ and for every adapted real sequence $ (s_k)_k
$
\begin{equation}\label{generic_condition}
 \lim_{k\to\iy}\noop{
 U(t^k_i)\circ A(\ell_k)\circ U(-t^k_i)\circ
M_{s_k}-A(\ell^i)\circ M_{s_k} }=0.
\end{equation}


\rm
 $ A $ satisfies
the \textit{character  condition} if for every properly converging
sequence  with perfect data $ (\pi_{\ell_k})_k$, $\ell_k\in S_N^{gen}$
 and for every limit point  $\pi_{\ell^i}, i\in C,$
and for every adapted real sequence $ (s_k)_k$
\begin{eqnarray}
\nn \lim_{k\to\iy}\noop{
 U(t^k_i)\circ A(\ell_k)\circ
U(-t^k_i)\circ M_{s_k\rho_i^k}-\nu(A(0))(i,k)\circ M_{s_k\rho_i^k} }=0.
\end{eqnarray}


$A$ satisfies the \textit{infinity}
condition, if for any properly converging sequence $
(\pi_{\ell_k})$, $\ell_k\in S_N ^{gen}$, with perfect data we have that
\begin{eqnarray}
\nn \lim_{k\to\iy} \noop{A( \ell_k)\circ M_{T_k}} &=&0,
\end{eqnarray}
where $ T_k={\mathbb R}\setminus\left(\bigcup_{i=1}^m[t_i^k-s_k\rho_i^k,t_i^k+s_k\rho_i^k]\right)$,  and that
for every sequence $ (\ell_k)_k\subset S_N^{gen} $, for which the
sequence of orbits $ G_N\cdot {\ell_k} $ goes to infinity we also have
$$ \lim_{k\to\infty}
A(\ell_k)=0. $$
 \end{definition}
We can now define the operator field $ C^* $-algebra $ D^*_N $,
which will be the image of the Fourier transform of $ C^*(G_N)$.

\begin{definition}\label{defcstarN}
\rm   Let $ D^*_N $ be the space of all bounded operator fields
$A= (A(\ell))\in \K(\l2\R), \ell\in S _N^{gen},A(0)\in C^*(\R^2)$,
such that $ A $ and the adjoint  field $ A^* $ satisfy the
generic, the character and the infinity conditions. Let for $ A\in
D^*_N $
\begin{eqnarray}
 \nn \no A\iy &:=&\sup\{\noop{A(\ell)}, \|A(0)\|_{C^*(\R^2)}: \ell\in S_N^{gen}\}.
\end{eqnarray}

 \end{definition}

It is clear that $ D^*_N $ is a Banach space for the
norm $
\no{\cdot}\iy $, since the   generic, the character and the infinity
conditions are stable for the sum, for scalar multiplication  and limits of
sequences of  operator fields.

\begin{theorem}\label{csatisfies}
Let $ a\in C^*(G_N) $ and let $ A $ be the operator field defined by
$ A=F( a) $ as in Definition \ref{fourN}.
 Then $ A $
satisfies the generic, the character and the infinity conditions.
 \end{theorem}
\begin{proof}
For the infinity condition, it suffices to remark that for any $
f\in L^1_c(G_N) $, and $k$ large enough, we have that $ \hat
f^2(s-t,t\cdot \ell_k)=0 $ for every $ s\in \R$, $t\in T_k $ and so
$ \pi_{\ell_k}(f)\circ M_{T_k}=0 $. If $
G_N\cdot \ell_k $ goes to infinity in the orbit space, then $
{\R\cdot\ell_k} $ is outside any given compact subset $ K\subset
\g_N^* $ and so $ \hat f^2(s-t,t\cdot \ell_k)=0, s,t\in\R $ and
hence $ \pi_{\ell_k}(f)=0 $ for $ k $ large enough. Using the
density argument, we see that the infinity
 condition is satisfied for every element in the Fourier transform of $
C^*(G_N) $.

For the generic condition,  let $ (\ell_k)_k $ be  a properly
converging sequence in $S_N$ with perfect data. Take $ i\in D $.
Then for an adapted sequence $ (s_k)_k $, $ f\in L^1_c(G_N) $ and $
\xi\in\lt, s\in\R$, we have that
\begin{eqnarray}\label{obs}
\nonumber &&  (U(t^k_i)\circ \pi_{\ell_k}(f)\circ U(-t^k_i)\circ
M_{s_k}-\pi_{\ell^i}(f)\circ
M_{s_k} )\xi(s) \\
 &=&  \int_{-s_k}^{s_k} \left(\hat f^2(s-t,(t+t_i^k)\cdot\ell_k))-\hat
f^2(s-t,t\cdot\ell^i))\right)\xi(t)dt.
\end{eqnarray}
Let $p_k$ and $q_i$ be the polynomials corresponding to $\ell_k$ and
$\ell^i$ respectively, i.e., $p_k(t)=\hat\ell_k(t)$ and
$q_i(t)=\hat\ell^i(t)$. Since
$\displaystyle\lim_{k\to\infty}\frac{s_k}{|t_i^k-a_j^k|}\to 0$,
$j\in J(i)$, there exists $R>0$ such that
$(s-t,(t+t_i^k)\cdot\ell_k)=(s-t,p_k(t+t_i^k),
-p_k'(t+t_i^k),\ldots)$ is out of the support of ${\hat f}^2$ if
$t\in[-s_k,s_k]$ and $|t|>R$. In fact
if $t\in[-s_k,s_k]$ we have
\begin{eqnarray*}
|p_k(t+t_i^k)|&=&|c_k\prod_{j=1}^d(t+t_i^k-a_j^k)|=|c_k\prod_{j\in
J(i)}|t_i^k-a_j^k|
\prod_{j\in J(i)}|\frac{t}{t_i^k-a_j^k}+1|\prod_{j\notin J(i)}|t+t_i^k-a_j^k|\\
 &\geq& |b_i|\prod_{j\in J(i)}|1-\frac{s_k}{|t_i^k-a_j^k|}|\prod_{j\notin
 J(i)}|t+t_i^k-a_j^k|,
\end{eqnarray*}
where $b_i$ is the leading coefficient of the polynomial $q_i$,
giving the statement.
Thus by (\ref{obs})
\begin{eqnarray*}
&&(U(t^k_i)\circ \pi_{\ell_k}(f)\circ U(-t^k_i)\circ
M_{s_k}-\pi_{\ell^i}(f)\circ
M_{s_k} )\xi(s) \\
 &=&  \int_{-R}^{R} \left(\hat
f^2(s-t,(t+t_i^k)\cdot\ell_k))-\hat
f^2(s-t,t\cdot\ell^i))\right)\xi(t)dt
\end{eqnarray*}
for $k$ large enough. It is clear now that $ U(t^k_i)\circ \pi_{\ell_k}(f)\circ
U(-t^k_i)\circ M_{s_k}-\pi_{\ell^i}(f)\circ M_{s_k}  $ converges
to $0 $ with respect to the Hilbert-Schmidt norm and hence in the
operator norm.

Let $a\in C^*(G_N)$. Then for any $\varepsilon>0$ there exists $f\in
L^1_c(G_N)$ such that $||\pi(f)-\pi(a)||_{op}\leq
||f-a||_{C^*(G_N)}<\varepsilon$ for any representation $\pi$ of
$C^*(G_N)$. Thus  for $A(\ell)=\pi_{\ell}(a)$, $\ell\in S_N^{gen}$
we have
\begin{eqnarray*}
&& \|U(t^k_i)\circ A(\ell_{k})\circ U(-t^k_i)\circ
M_{s_k}-A(\ell^i)\circ M_{s_k}\|_{\text{op}} =\|U(t^k_i)\circ
(A(\ell_{k})-\pi_{l_k}(f))\circ U(-t^k_i)\|_{\text{op}}\\ && +
\|U(t^k_i)\circ \pi_{\ell_k}(f)\circ U(-t^k_i)\circ
M_{s_k}-\pi_{\ell^i}(f)\circ
M_{s_k}\|_{\text{op}}+\|(\pi_{\ell^i}(f)-A(\ell^i))\|_{\text{op}}\to
0,
\end{eqnarray*}
and hence  $A$ satisfies the generic condition.

Choose now $ i\in C $. By (\ref{seqik}), for $ k\in \N,
s\in\R,\xi\in\lt,f\in L_c^1(G_N) $
\begin{eqnarray}
\nn && U(t^k_i)\circ \pi_{\ell_k}( f)\circ U(-t^k_i)\circ
M_{s_k\rho_i^k}-\nu (F({f}(0)))(i,k)\circ
M_{s_k\rho_i^k} (\xi)(s)\\
\nn &=&\int_{-s_k\rho_i^k}^{s_k\rho_i^k}(\int_{\R}\hat
f^2(s-t,(t+t^k_i)\cdot \ell_k))-\hat
f^2(s-t,-\frac{b}{s_k}+p^i(\frac{t}
{\rh^k_i}),0\ldots,0)\\
&&\nn
 \overline{\et(b)} \et(s_k(p^i(\frac{s}{\rh^k_i})
-p^i(\frac{t} {\rh^k_i}))+b)db)\xi(t) dt+\\
&&\nn
\int_{-s_k}^{s_k}(\int_{\R}\hat f^2(s-t,(t+t^k_i)\cdot
\ell_k))
 \overline{\et(b)} (\eta(b)-\et(s_k(p^i(\frac{s}{\rh^k_i})
-p^i(\frac{t} {\rh^k_i}))+b))db)\xi(t) dt.
\end{eqnarray}
In order to show that
\begin{equation}\label{char}
\|U(t^k_i)\circ \pi_{\ell_k}( f)\circ U(-t^k_i)\circ
M_{s_k\rho_i^k}-\nu (F({f}(0)))(i,k)\circ M_{s_k\rho_i^k}\|\to
0, \  k\to\infty,
\end{equation}
 consider
\begin{eqnarray*}
q(k,i)(s,b)&=&\int_{-s_k\rho_i^k}^{s_k\rho_i^k}(\hat
f^2(s-t,(t+t^k_i)\cdot \ell_k))-\hat
f^2(s-t,-\frac{b}{s_k}+p^i(\frac{t}
{\rh^k_i}),0\ldots,0))\\
&&\et(s_k(p^i(\frac{s}{\rh^k_i}) -p^i(\frac{t}
{\rh^k_i}))+b))\xi(t) dt=u(k,i)+v(k,i),
\end{eqnarray*}
where
\begin{eqnarray*}
u(k,i)(s,b)&=&\int_{-s_k\rho_i^k}^{s_k\rho_i^k}(\hat
f^2(s-t,p_k(t+t^k_i), -p_k'(t+t^k_i),\ldots)-\hat
f^2(s-t,p_k(t+t^k_i),0,\ldots))\\
&&\et(s_k(p^i(\frac{s}{\rh^k_i}) -p^i(\frac{t}
{\rh^k_i}))+b))\xi(t) dt,\\
v(k,i)(s,b)&=&\int_{-s_k\rho_i^k}^{s_k\rho_i^k}(\hat
f^2(s-t,p_k(t+t^k_i),0,\ldots)-\hat
f^2(s-t,-\frac{b}{s_k}+p^i(\frac{t}
{\rh^k_i}),0\ldots))\\
&&\et(s_k(p^i(\frac{s}{\rh^k_i}) -p^i(\frac{t}
{\rh^k_i}))+b))\xi(t) dt.
\end{eqnarray*}
and let
$$w(k,i)(s)=\int_{\R}\int_{-s_k}^{s_k}(\int_{\R}\hat f^2(s-t,(t+t^k_i)\cdot
\ell_k))\overline{\eta(b)} (\eta(b)-\et(s_k(p^i(\frac{s}{\rh^k_i})
-p^i(\frac{t} {\rh^k_i}))+b))db)\xi(t) dt.$$

Our aim is to prove that for $p(s,b)=1_{{\mathbb R}\times\supp\eta}(s,b)$
\begin{equation}\label{conv}
\|u(k,i)p\|_2\leq \omega_k||\xi||_2,
\|v(k,i)p\|_2\leq\delta_k||\xi||_2\text{ and }\|w(k,i)\|_2\leq r_k||\xi||_2
\end{equation}
with $\omega_k,\delta_k, r_k\to 0$ as $k\to\infty$. This will imply $$\int_{\R}|
\int_{\R}
(q(k,i)(s,b))\overline{\eta(b)}db|^2ds\leq\|q(k,i)\|_2^2\|\eta\|_2\leq
(\omega_k+\delta_k)^2\|\xi\|_2^2$$
which together with $\|w(k,i)\|_2\leq r_k||\xi||_2$ will give
(\ref{char}).

To see this  we note first that since
$\frac{s_k\rho_i^k}{|a_k^j-t_i^k|}\to 0$ if $j\notin L(i)$,
we have that for $|t|\leq s_k\rho_i^k$
\begin{eqnarray*}
|p_k(t+t_i^k)|&=&|c_k\prod_{j=1}^d(t+t_i^k-a_j^k)|=|c_k\prod_{j}
|t_i^k-a_j^k|
\prod_{j\notin L(i)}|\frac{t}{t_i^k-a_j^k}+1|
\prod_{j\in L(i)}|\frac{t}{t_i^k-a_j^k}+1)|\\
 &\geq& \sigma\prod_{j\notin
L(i)}|1-\frac{s_k\rho_i^k}{|t_i^k-a_j^k|}|\prod_{j\in L(i)}
|\frac{|t|}{\rho_i^k}\frac{\rho_i^k}{|t_i^k-a_j^k|}-1|
\end{eqnarray*}
for some $\sigma>0$. Thus for large $k$ there exists $R>0$ such
that
$\hat{f}^2(s-t,p_k(t+t_i^k),-p_k'(t+t_i^k),\ldots)=0$
and $\hat{f}^2(s-t,p_k(t+t_i^k),0,\ldots)=0$ if $|t|<s_k\rho_i^k$
and $|t|>R\rho_i^k$. Hence the integration over the
interval $[-s_k,s_k]$ can be replaced by the integration over
$[-R\rho_i^k,R\rho_i^k]$ in the expression for $u(k,i)$,
$v(k,i)p$ and $w(k,i)$. Since $f\in L^1_c(G_N)$ we have that
$$|\hat{f}^2(s-t,p_k(t+t_i^k),0,\ldots)-\hat{f}^2(s-t,p^i(\frac{t}{\rho_i^k}
)-\frac{b}{s_k},0,\ldots)|\leq
C|p_k(t+t_i^k)-p^i(\frac{t}{\rho_i^k})+\frac{b}{s_k}|\frac{1}{1+|t-s|^m}$$ for
some constant $C>0$ and $m\in {\mathbb N}$, $m\geq 2$ . This gives
\begin{eqnarray*}
&&\|v(k,i)p\|_2^2
=C\int_{\R^2}\left
|\int_{-R\rho_i^k}^{R\rho_i^k}\eta(s_k(p^i(\frac{s}{\rho_i^k})-
p^i(\frac{t}{\rho_i^k})+b)\right.\\
&&\quad\quad\left.(p_k(t+t_i^k)-p^i(\frac{t}{\rho_i^k})+\frac{b}{s_k}
)\frac{\xi(t)}{1+|t-s|^m}dt\right |^2ds db\\
&&\leq
\frac{3C}{s_k^2}\int_{\R^2}\int_{-R\rho_i^k}^{R\rho_i^k}\left
|\eta(s_k(p^i(\frac{s}{\rho_i^k})-
p^i(\frac{t}{\rho_i^k})+b)\right.\\
&&\quad\quad\left.(b+s_k(p^i(\frac{s}{\rho_i^k})-
p^i(\frac{t}{\rho_i^k}))\frac{\xi(t)}{1+|t-s|^m}dt \right|^2ds db\\
&&+3C\int_{\R^2}\int_{-R\rho_i^k}^{R\rho_i^k}\left
|\eta(s_k(p^i(\frac{s}{\rho_i^k})-
p^i(\frac{t}{\rho_i^k}))+b)
(p_k(t+t_i^k)-p^i(\frac{t}{\rho_i^k}))\frac{\xi(t)}{
1+|t-s|^m}dt\right|^2 ds db\\
&&+ 3C\int_{\R^2}\int_{-R\rho_i^k}^{R\rho_i^k}\left
|\xi(t)\eta(s_k(p^i(\frac{s}{\rho_i^k})-
p^i(\frac{t}{\rho_i^k}))+b)(p^i(\frac{t}{\rho_i^k})-
p^i(\frac{s}{\rho_i^k}))\frac{1}{1+|t-s|^m}dt\right|^2ds db\\
&&\leq
\frac{C_1}{s_k^2}\int_{\R^2}\int_{-R\rho_i^k}^{R\rho_i^k}\left
|\xi(t)\tilde\eta(s_k(p^i(\frac{s}{\rho_i^k})-
p^i(\frac{t}{\rho_i^k}))+b)\right|^2\frac{1}{1+|t-s|^m}dt ds db\\
&&\quad\quad(\text{where }\tilde\eta(b)=b\eta(b))\\
&&+C_1\int_{\R^2}\int_{-R\rho_i^k}^{R\rho_i^k}\left
|\xi(t)\eta(s_k(p^i(\frac{s}{\rho_i^k})-
p^i(\frac{t}{\rho_i^k}))+b)(p_k(t+t_i^k)-p^i(\frac{t}{\rho_i^k}))\right|^2\\
&&\quad\quad\frac{1}{1+|t-s|^m}dt ds db\\
&&+ \int_{\R^2}\int_{-R\rho_i^k}^{R\rho_i^k}\left
|\xi(t)\eta(s_k(p^i(\frac{s}{\rho_i^k})-
p^i(\frac{t}{\rho_i^k}))+b)(p^i(\frac{t}{\rho_i^k})-
p^i(\frac{s}{\rho_i^k}))\right|^2 \\
&&\quad\quad\frac{1}{1+|t-s|^m}dt ds db
\end{eqnarray*}
\begin{eqnarray*}
&\leq&
\frac{C_2}{s_k^2}\|\tilde\eta\|_2^2\|\xi\|_2^2+C_2\|\eta\|_2^2\int_{-R}^{R}
|\xi(t\rho_i^k)|^2|p_k(t\rho_i^k+t_i^k)-p^i(t)|^2\rho_i^kdt\\
&+&C_3\|\eta\|_2^2\int_{\R}\int_{-R\rho_i^k}^{R\rho_i^k}
|\xi(t)|^2\left|p^i(\frac{t}{\rho_i^k})-p^i(\frac{s}{\rho_i^k})\right|^2\frac{1}
{1+|t-s|^m}dtds
\end{eqnarray*}
As $p_k(t\rho_i^k+t_i^k)-p^i(t)$ converges to $0$ uniformly on
each compact,
$$\int_{-R}^R|\xi(t\rho_i^k)|^2|p_k(t\rho_i^k+t_i^k)-p^i(t)|^2\rho_i^kdt\leq
r_k\|\xi\|_2^2$$ with $r_k\to 0$ as $k\to \infty$. Moreover,
$\displaystyle
p^i(\frac{t}{\rho_i^k})-p^i(\frac{s}{\rho_i^k})=\frac{t-s}{\rho_i^k}\sum_l
\alpha_l(\frac{t}{\rho_i^k})\beta_l(\frac{t-s}{\rho_i^k})$ for
some finite number of polynomials $\alpha_l$, $\beta_l$ which do
not depend on $k$. Thus
\begin{eqnarray*}&&\int_{\R}\int_{-R\rho_i^k}^{R\rho_i^k}
|\xi(t)|^2\left|p^i(\frac{t}{\rho_i^k})-p^i(\frac{s}{\rho_i^k})\right|^2\frac{1}
{1+|t-s|^m}dtds\\
&=&\int_{\R}\int_{-R\rho_i^k}^{R\rho_i^k}|\xi(t)|^2\left|\frac{t-s}{\rho_i^k}
\sum_l
\alpha_l(\frac{t}{\rho_i^k})\beta_l(\frac{t-s}{\rho_i^k})\right|^2\frac{1}{
1+|t-s|^m}dtds\\
&\leq&
\frac{C_4}{(\rho_i^k)^2}\int_{-R\rho_i^k}^{R\rho_i^k}|\xi(t)|^2dt\leq\frac{C_4}{
(\rho_i^k)^2}\|\xi\|_2^2
\end{eqnarray*}
for a properly chosen $m$. It follows now that
$$||v(k,i)p||_2\leq \delta_k\|\xi\|_2$$
for some $\delta_k\to 0$ as $k\to\infty$.

For $w(k,i)$ we have
\begin{eqnarray*}
\|w(k,i)\|^2&=& \int_{\R}\left|\int_{-R\rho_i^k}^{R\rho_i^k}\int_R
\hat
f^2(s-t,(t+t_i^k)\cdot\ell_k)\overline{\eta(b)}\right.\\
&&\left.(\eta(b)-\et(s_k(p^i(\frac{s}{\rh^k_i})
-p^i(\frac{t} {\rh^k_i}))+b))\xi(t) dbdt\right|^2ds\\
&\leq& C\|\eta\|^2_2\int_{\R}\left|\int_{-R\rho_i^k}^{R\rho_i^k}|s_k(p^i(\frac{s}{\rh^k_i})
-p^i(\frac{t} {\rh^k_i}))|\frac{1}{1+|t-s|^m}|\xi(t)| dt\right|^2ds\\
&\leq& C\|\eta\|^2_2\int_R\int_{-R\rho_i^k}^{R\rho_i^k} |s_k(p^i(\frac{s}{\rh^k_i})
-p^i(\frac{t} {\rh^k_i}))|^2\frac{1}{1+|t-s|^m}|\xi(t)|^2 dtds
\end{eqnarray*}
for some constant $C$. Then using the previous arguments we get
$$\|w(k,i)\|^2\leq \frac{Ds_k^2}{(\rho_i^k)^2}\|\xi\|^2_2\|\eta\|^2_2.$$
As $\frac{s_k}{\rho_i^k}\to 0$ we get the desired inequality for $w(k,i)$.

To prove the  inequality for $u(k,i)$ we have as in the previous
case that
\begin{eqnarray*}
&&|\hat{f}^2(s-t,p_k(t+t_i^k),-p_k'(t+t_i^k),\ldots)-\hat{f}^2(s-t,p_k(t+t_i^k),0
,\ldots)|\\&&\leq C(\sum_{n=1}^{N-2}
|p_k^{(n)}(t+t_i^k)|^2)^{1/2}\frac{1}{1+|t-s|^m} \end{eqnarray*} for
some constant $C>0$ and $m\in {\mathbb N}$, $m\geq 2$; here
$p_k^{(n)}$ denotes the $n$-th derivative of $p_k$. For $n=1$ we
have
$$|p_k'(t+t_i^k)|=|c_k\prod_{j}(t_i^k-a_k^j)|\sum_l\frac{1}{(t_i^k-a_k^l)}\prod
_{j\ne l}
(\frac{t}{t_i^k-a_k^j}+1) |\leq
\sigma\frac{1}{\rho_i^k}\left(\frac{|t|}{\rho_i^k}+1\right)^{d-1}$$
for some constant $\sigma>0$. Similar
inequalities hold  for higher order derivatives
$p_k^{(n)}(t+t_i^k)$ which show that
\begin{eqnarray*}
&&\|u(k,i)\|_2^2=\\ &=&
\int_{\R^2}|\int_{-R\rho_i^k}^{R\rho_i^k}(\hat
f^2(s-t,p_k(t+t^k_i), -p_k'(t+t^k_i),\ldots)-\hat
f^2(s-t,p_k(t+t^k_i),0,\ldots))\\
&& \et(s_k(p^i(\frac{s}{\rh^k_i}) -p^i(\frac{t}
{\rh^k_i}))+b))\xi(t)
dt|^2dbds\\
&\leq& \int_{\R^2}|\int_{-R\rho_i^k}^{R\rho_i^k}C(\sum_{n=1}^{N-2}
|p_k^{(n)}(t+t_i^k)|^2)^{1/2}\frac{1}{1+|t-s|^m}\et(s_k(p^i(\frac{s}{\rh^k_i})
-p^i(\frac{t} {\rh^k_i}))+b))\xi(t)
dt|^2dbds\\
&\leq&
\frac{1}{\rho_i^k}\int_{\R^2}|\int_{-R\rho_i^k}^{R\rho_i^k}|p\left(\frac{|t|}{
\rho_i^k}\right)|\frac{1}{1+|t-s|^m}\et(s_k(p^i(\frac{s}{\rh^k_i})
-p^i(\frac{t} {\rh^k_i}))+b))\xi(t)
dt|^2dbds\\
&\leq &\frac{C'}{\rho_i^k}\|\eta\|_2^2\|\xi\|_2^2
\end{eqnarray*}
for a polynomial $p$.
Thus we get the required inequality for $u(k,i)$ and hence
\begin{eqnarray}
 \nn \lim_{k\to\iy}\noop{ U(t^k_i)\circ \pi_{\ell_k}( f)\circ U(-t^k_i)\circ
M_{s_k\rho_i^k}-\nu (F({f}(0)))(i,k)\circ M_{s_k\rho_i^k}}
&=&0.
\end{eqnarray}

To show now that the character condition holds for the fields
$A\in \widehat {C^*(G_N)}$ we use again
 the  density of $ L^1_c(G_N) $ in $ C^*(G_N) $.

\end{proof}

\begin{corollary}\label{ahom}
Let $ (\pi_{\ell_k})_k $ be a properly converging sequence in $
\widehat G_N $ with perfect data $ ((t^k_i)_k, (\rh_i^k), (s^k_i)) $.
 Let $ i\in C $. Then for every $ \va,\ps\in C^*(\R^2) $ we have that
\begin{eqnarray}
 \nn \lim_{k\to\iy}\noop{\nu(\va)(i,k)\circ \nu(\ps)(i,k)-\nu(\va\ps)(i,k)}
&=&0.
\end{eqnarray}

 \end{corollary}
\begin{proof} Indeed, if we take
first $ \va,\ps $ in $ \S(\R^2) $, then we can choose $ f,g \in
\S(G_N)$, such that $ \rho(f)=\va,\rho(g)=\ps $ and
so, by Proposition \ref{pikjcom} and Theorem~\ref{csatisfies}, \begin{eqnarray}
\nonumber && \noop{\nu(i,k)(\va)\circ \nu(i,k)(\ps)-\nu(i,k)(\va\ps)} \\
\nonumber&\leq&\noop{(\nu(i,k)(\va)\circ
\nu(i,k)(\ps)-\nu(i,k)(\va\ps))\circ M_{s_k\rho_i^k} }\\
\nonumber&+&\noop{(\nu(i,k)(\va)\circ
\nu(i,k)(\ps)-\nu(i,k)(\va\ps))\circ (\I-M_{s_k\rho_i^k})
}\\
\nonumber &\leq&  \noop{(\nu(i,k)(\va)\circ \nu(i,k)(\ps)\\
\nn  &&-
(U(t^k_i)\circ\pi_{\ell_k}(f)\circ U(-t^k_i))\circ( U(t^k_i)\circ
\pi_{\ell_k}(g)\circ U(-t^k_i)))\circ M_{s_k\rho_i^k}
}\\
\nonumber &+&\noop{(U(t^k_i)\circ\pi_{\ell_k}(f\ast g)\circ
U(-t^k_i)-\nu(i,k)(\va\ps)) \circ M_{s_k\rho_i^k}}\\
\nonumber&+&\noop{(\nu(i,k)(\va)\circ
\nu(i,k)(\ps)-\nu(i,k)(\va\ps))\circ (\I-M_{s_k\rho_i^k})
}\\
\nonumber &\leq&  \noop{(\nu(i,k)(\va)-
(U(t^k_i)\circ\pi_{\ell_k}(f)\circ U(-t^k_i))\circ(\I-M_{s_k\rho_i^k})
\circ \nu(i,k)(\ps)\circ M_{s_k\rho_i^k}}\\
\nonumber&+&  \noop{
(\nu(i,k)(\va)-
(U(t^k_i)\circ\pi_{\ell_k}(f)\circ U(-t^k_i))\circ M_{s_k\rho_i^k}\circ
\nu(i,k)(\ps)\circ M_{s_k\rho_i^k}}\\
\nonumber &+&  \noop{ ((U(t^k_i)\circ\pi_{\ell_k}(f)\circ
U(-t^k_i))\circ( \nu(i,k)(\ps)- ( U(t^k_i)\circ
\pi_{\ell_k}(g)\circ U(-t^k_i))\circ M_{s_k\rho_i^k}}\\
\nonumber &+&\noop{(U(t^k_i)\circ\pi_{\ell_k}(f\ast g)\circ
U(-t^k_i)-\nu(i,k)(\va\ps))\circ M_{s_k\rho_i^k} }\\
\nonumber&+&\noop{(\nu(i,k)(\va)\circ
\nu(i,k)(\ps)-\nu(i,k)(\va\ps))\circ(\I-M_{s_k\rho_i^k})
}\to 0\text{ as }k\to\infty.
\end{eqnarray}
The usual density condition shows that the statement holds for all  $\varphi$, $\psi\in C^*(\R^2)$.
 \end{proof}

\begin{theorem}
The space $ D^*_N $ is  a $ C^* $-algebra, which is isomorphic
with $ C^*(G_N) $ for every $ N\in \N, N\geq 3 $.
 \end{theorem}

\begin{proof}
Let us first show that $ D^*_N $ is  a $ C^* $-algebra. We prove
first that $D^*_N$ is closed under multiplication. Let
$A=(A(\ell), \ell\in S_N)$ and $B=( B(\ell),\ell\in S_N)$ satisfy
the generic condition and let $(\pi_{\ell_k})_{k}\subset \hat {G}_N$ be
a properly convergent sequence with perfect data such that for
every limit point $\pi_{l^i}$, $i\in D$, and for every adapted
real sequence $(s_k)_k$ the fields $A$, $B$ satisfy
(\ref{generic_condition}). Then
\begin{eqnarray*}
 &&\noop{
 U(t^k_i)\circ A(\ell_k)\circ B(\ell_k)\circ U(-t^k_i)\circ
M_{s_k}-A(\ell^i)\circ B(\ell^i)\circ M_{s_k} }\\ &&\leq \noop{
(U(t^k_i)\circ A(\ell_k)\circ U(-t_i^k)\circ M_{s_k}-A(\ell^i)\circ
M_{s_k})\circ U(t_i^k)\circ B(\ell_k)\circ U(-t^k_i)\circ
M_{s_k}}\\
&&+\noop{A(\ell^i)\circ M_{s_k}\circ (U(t_i^k)\circ B(\ell_k)\circ
U(-t_i^k)\circ M_{s_k}-B(\ell^i)\circ M_{s_k})}\\
&&+\noop{U(t_i^k)\circ A(\ell_k)\circ U(-t_i^k)\circ (\I-
M_{s_k})\circ(U(t_i^k)\circ B(\ell_k)\circ U(-t_i^k)\circ
M_{s_k}-B(\ell^i)\circ M_{s_k})}
\\
&&+\noop{A(\ell^i)\circ(\I-M_{s_k})\circ B(\ell^i)\circ
M_{s_k}}\\
&&+\noop{U(t_i^k)\circ A(\ell_k)\circ U(-t_i^k)\circ
(\I-M_{s_k})\circ B(\ell^i)\circ M_{s_k}}.
\end{eqnarray*}
Since $B(\ell^i)$ is compact and $\I-M_{s_k}$ converges to $0$
strongly, $\noop{(\I-M_{s_k})\circ B(\ell^i)}\to 0$ giving that the
product $A(\ell)\circ B(\ell)$ satisfies the generic condition.

To see that the character condition is closed under multiplication
we argue as before, but use  $\noop{(\I-M_{s_k\rho_i^k})\circ
\nu(\va)(i,k)\circ M_{s_k\rho_i^k}}\to 0$ which is due to
Propsition~\ref{pikjcom}.


The infinity condition is clearly closed under multiplication of
fields.

By Theorem \ref{csatisfies}, the Fourier transform $ F $ maps
 $C^*(G_N)$ into $
D^*_N $
Let us show that the Fourier$ F $  is also onto. By the Stone-Weerstrass
approximation theorem, we must only prove  that  the dual space of $
D^*_N $ is the
same as the dual space of $ C^*(G_N) $. We proceed by induction on $ N $. If $
N=3 $, then $
G_N $ is the Heisenberg group and the statement  follows  from
Theorem~\ref{finalres}. Let $ \pi\in \widehat
{D^*_N} $. 

Let for $
M=3,\cdots  N-1 $, $ R_M: D_N^*\to D_M^*$ be the restriction map, i.e. denote
 by $ q_M:\g_N\to \g_N/\b_{N-M}\simeq \g_M $ the quotient map and by
  $q_M^t:\g_{M}^*\simeq\b_{N-M}^\perp\to\g_N^* $ its transpose. Then for an
operator
field $ A\in D^*_N $ we define the operator field $ R_M(A) $ over $ S_M^{\rm gen
} $ by:
\begin{eqnarray}
 \nn R_M (A)(\tilde \ell):= A(q^t_M(\tilde \ell)), \tilde \ell \in S^{\rm
gen}_M.
\end{eqnarray}
It follows from the definition of $ D_N^* $ that the image of $ R_M $ is
contained in $ D_M^* $.
Hence  $ R_M $ is  a homomorphism of $ C^* $-algebras, whose kernel
$ I_M $ is the ideal $$ I_M=\{A\in D^*_N, A(\ell)=0 \text{ for all
}\ell\in S_N \cap \b_{N-M}^\perp\} .$$
Let $ Q_M:C^*(G_N)\to C^*(G_M)\simeq C^*(G_N/B_{N-M}) $ be the canonical
projection. Then the kernel of this projection is the ideal $ J_M:=\{a\in
C^*(G_N); \pi_\ell(a)=0, \ell\in S_N\cap \b_{N-M}^\perp\} $.
Let us write $ F_M $ for the Fourier transform $ C^*(G_M)\to D^*_M $. With
these notations we have the formula
\begin{eqnarray}\label{modgm}
  R_M(F_N(a))=F_M(a\text{ modulo }J_M), a\in C^*(G_N).
\end{eqnarray}
Since by the induction hypothesis $ \widehat{D^*_M}=F_M(\widehat{C^*(G_M)}) $ 
for every $3\leq
M\leq N-1 $ we see from (\ref{modgm}) that $ R_M(F_N(C^*(G_N)))=F_M(C^*(G_M))=D_M^
*$ and so the mapping $ R_ M $ is surjective for
such an $M$. Hence $ D^*_N/I_M\simeq C^*(G_M) $.  We have also
$I_{N-1}\subseteq I_{N-2}\subseteq\ldots\subseteq I_3$.

If $ \pi(I_{N-1})=\{0\} $ then $\pi\in \widehat{G_{N}/B_1}\subset \widehat{G_N}
$.

Suppose now that $ \pi(I_{N-1})\ne \{0\} $. Let us show that $ I_{N-1}\simeq
 C_0(S_N^1,\K(L^2(\R)))$. It is clear from the
definition of $D_N^*$ that $C_0(S_N^{1},\K(\lt))\subset D_N^*$ and so is
contained in $
I_{N-1}$. It suffices to show now that $ I_{N-1}\subset
 C_0(S_N^1,\K(L^2(\R)))$. For that it is enough to see that for any  element $ A
$   in $ I_{N-1} $ and any  sequence  $ {(\ell_k)_k} $ in $ S_N^1 $ for which
either  $ (\pi_{\ell_k}) $ converges  to infinity or to a representation $
\pi_\ell $ with $ \ell\not\in S_N^1$, we have that  $ \lim_k\noop
{A(\ell_k)}=0 $. This follows from the infinity condition in the first case.
In the second case no limit point of the sequence $ (\pi_{\ell_k}) $ is in $
S_N^1 $ by Remark \ref{continorm}. It suffices to show then that $ \lim_k
\noop{A(\ell_k)}=0 $ for every subsequence with perfect data (also indexed by $
k $ for simplicity of notation).
We have with the notations  of \ref{gencon} 
that for $k\in\N $
\begin{eqnarray} 
 \nn A(\ell_k) &=& A(\ell_k)\circ M_{S_k} +
A(\ell_k)\circ M_{T_k}.
\end{eqnarray}
 where $ S_k=\cup_{i} (t^k_i-s_k\rho_i^k,t^k_i+s_k\rho_i^k) $, $T_k=\R\setminus
S_k$.

Since $ A(\ell)=0 $ for every $ \pi_{\ell} $ in the limit set of the
sequence $ (\pi_{\ell_k})_k $,  the generic and the
character  conditions say that
$$ \lim_k\noop{U(t^k_i)\circ A(\ell_k)\circ U(t^k_i)\circ M_{s_k\rho_i^k}}=0 .$$

Hence
$$\lim_k\noop{ A(\ell_k)\circ M_{t^k_i,s_k\rho_i^k}}=0 $$

and since also $$\lim_k\noop{ A(\ell_k)\circ M_{T_k}}=0 $$ it
follows that $ \lim_k \noop{A(\ell_k)}=0. $  Hence $ I_{N-1}\subset
 C_0(S_N^1,\K(L^2(\R)))$ and so $ I_{N-1}=
 C_0(S_N^1,\K(L^2(\R)))$. Finally $ \pi\res {I_{N-1}} $ is evaluation in some
point $ \ell\in S_N^1 $ and so $ \pi\in \widehat{G_N} $. This finishes the
proof of the theorem.

\end{proof}

\bigskip

\noindent {\bf Acknowledgements.}  We would like
to thank K. Juschenko for the reference \cite{gorbachev}.
The second author was supported by the Swedish Research Council.

 \noi Jean Ludwig, {\it Laboratoire LMAM, UMR 7122,   D\'epartement de
Math\'ematiques,
Universit\'e Paul Verlaine  Metz, Ile de Saulcy, F-57045 Metz
cedex 1, France, ludwig@univ-metz.fr}\\

\noi Lyudmila  Turowska, {\it  Department of Mathematics, Chalmers University of
Technology and University of Gothenburg, SE-412 96 G\"oteborg, Sweden, turowska@chalmers.se}

\begin{thebibliography}{99}


\bibitem[AKLSS] {A.K.L.S.S.} Archbold, R. J.; Kaniuth, E.; Ludwig, J.;
Schlichting, G.; Somerset, D. W. B. Strength of convergence in
duals of $C^*$-algebras and nilpotent Lie groups. Adv. Math.  \textbf{158}
(2001),  no. 1, 26--65.
\bibitem[ALS] {A.L.S.} Archbold, R. J.; Ludwig, J.; Schlichting, G. Limit
sets and strengths of convergence for sequences in the duals of
thread-like Lie groups. Math. Z. \textbf{255} (2007),  no. 2, 245--282.

\bibitem[CG]{CG} Corwin, L.J.; Greenleaf, F.P. Representations of nilpotent Lie
groups and their applications. Part I. Basic theory and examples. Cambridge
Studies in Advanced Mathematics, 18. Cambridge University Press, Cambridge,
1990. viii+269 pp.
\bibitem[De]{delaroche} Delaroche, C. Extensions des $C^*$-alg\'ebres. (French)
Bull. Soc. Math. France M\'em., No. 29. Suppl\'ement au Bull. Soc. Math.
France, Tome 100. Soci\'et\'e
Math\'ematique de France, Paris, 1972. 142 pp.
\bibitem [Di]{Di} Dixmier, J. $C^*$-algebras. Translated from the French
by Francis Jellett. North-Holland Mathematical Library, Vol. 15.
North-Holland Publishing Co., Amsterdam-New York-Oxford, 1977.
xiii+492 pp.
\bibitem [Gor]{gorbachev} Gorbachev, N. V. $C^{*} $-algebra of the Heisenberg group.
 Uspekhi Mat. Nauk  \textbf{35}  (1980), no. 6(216), 157--158.
\bibitem[Lee1]{Lee1} Lee, R-Y. On the $C^*$ Algebras of Operator Fields Indiana
University Mathematics Journal, {\bf 26} No. 2 (1977), 351-372.

\bibitem[Lee2]{Lee2} Lee, R-Y. Full Algebras of Operator Fields Trivial Except
at One Point. Indiana University Mathematics Journal,  {\bf 25}
No. 4 (1976), 303-314.
\bibitem [LRS] {LRS} Ludwig,J; Rosenbaum, G.; Samuel,
J.
The elements of bounded trace in the $C^*$-algebra of a nilpotent Lie
group.
Invent. Math. \textbf{83} (1985), no. 1, 167--190.
\bibitem[L]{L} Ludwig, J. On the behaviour of sequences in the dual of a nilpotent Lie group. Math. Ann. {\bf 287}, 239-257 (1990).
\bibitem[R]{rosenberg} Rosenberg, J.
Homological invariants of extensions of $C^*$-algebras. Operator algebras and applications, Part 1 (Kingston, Ont., 1980), pp. 35--75,
Proc. Sympos. Pure Math., 38, Amer. Math. Soc., Providence, RI, 1982.
\bibitem[W]{wegge-olsen} Wegge-Olsen, N. E.
$K$-theory and $C^*$-algebras.
A friendly approach. Oxford Science Publications. The Clarendon Press, Oxford University Press, New York, 1993.



\end{thebibliography}
\end{document}